\documentclass{amsart}
\usepackage{amsmath,amsthm,amssymb,amscd}
\usepackage{graphicx}
\usepackage{tikz}
\usepackage[german,frenchb,english]{babel}
\usepackage{hyperref}
\usepackage{mathrsfs}
\usepackage{todonotes}
\usepackage[all,cmtip]{xy}
\input diagxy    %%%%%%%%% it's this line which causes the compilation problems, try switch to \usepackage or \input (or \import or \include
\usepackage[margin=1.2in]{geometry}

\newcommand{\sSet}{\mathsf{sSet}}
\newcommand{\ssSet}{\mathsf{ssSet}}

\DeclareMathOperator{\ob}{ob}

\DeclareMathOperator{\Tor}{Tor}

\newcommand{\Ddelta}{\Delta}

\newcommand{\Qb}{\mathbb{Q}}

\newcommand{\into}{\hookrightarrow}

\newcommand{\onto}{\twoheadrightarrow}

\newcommand{\Inv}{\mathsf{Inv}^{\textit{a}}}

\DeclareMathOperator{\A}{\mathsf{A}}
\DeclareMathOperator{\B}{\mathsf{B}}

\DeclareMathOperator{\E}{\mathcal{E}}
\DeclareMathOperator{\Idx}{\mathbb{I}dx}

\DeclareMathOperator{\fPoset}{poSet^{fr,filt}_{f}}
\DeclareMathOperator{\bSet}{poSet^{filt}_f}
\DeclareMathOperator{\bcSet}{poSet^{filt}}

\DeclareMathOperator{\Fb}{\mathbb{F}}

\DeclareMathOperator{\Lex}{Lex}

\DeclareMathOperator{\Prok}{\mathsf{Pro}^{\textit{a}}_{\kappa}}

\newcommand{\Lc}{\mathcal{L}}

\newcommand{\Excat}{\mathsf{Cat}_{\textit{ex}}}
\newcommand{\Expair}{\mathsf{Cat}_{\textit{ex}}^{\textit{pair}}}

\newcommand{\Index}{\mathsf{Index}}

\newcommand{\Spaces}{\mathsf{Spaces}}

\DeclareMathOperator{\Vect}{\mathsf{Vect}}

\DeclareMathOperator{\Indk}{\mathsf{Ind}^{\textit{a}}_{\kappa}}
\DeclareMathOperator*{\colim}{\varinjlim}

\newcommand{\C}{\mathsf{C}}
\newcommand{\Cc}{\mathsf{C}} % because C vs Cc seem to have been almost randomly mixed during the writing process...

\newcommand{\Gr}{Gr}

\newcommand{\Ab}{Ab}

\DeclareMathOperator{\id}{id}
\DeclareMathOperator{\grp}{\times}

\DeclareMathOperator{\Set}{Set}

\DeclareMathOperator{\Pro}{\mathsf{Pro}^{\textit{a}}}

\DeclareMathOperator{\Ind}{\mathsf{Ind}^{\textit{a}}}
\DeclareMathOperator{\Fun}{Fun}

\DeclareMathOperator{\op}{op}

\DeclareMathOperator{\Aut}{Aut}

\DeclareMathOperator{\GL}{GL}

\newcommand{\Zb}{\mathbb{Z}}
\DeclareMathOperator{\D}{\mathsf{D}}

\DeclareMathOperator{\Kk}{\mathcal{K}}

\DeclareMathOperator{\Tate}{\mathsf{Tate}}
\DeclareMathOperator{\elTate}{\mathsf{Tate}^{\textit{el}}}

\begin{document}
\newtheorem{definition}{Definition}[section]
\newtheorem{theorem}[definition]{Theorem}
\newtheorem{proposition}[definition]{Proposition}
\newtheorem{corollary}[definition]{Corollary}
\newtheorem{conj}[definition]{Conjecture}
\newtheorem{lemma}[definition]{Lemma}
\newtheorem{rmk}[definition]{Remark}
\newtheorem{cl}[definition]{Claim}
\newtheorem{example}[definition]{Example}
\newtheorem{problem}[definition]{Problem}
\newtheorem{claim}[definition]{Claim}
\newtheorem{ass}[definition]{Assumption}
\newtheorem{warning}[definition]{Warning}
\newtheorem{porism}[definition]{Porism}
\newtheorem{notation}[definition]{Notation}

\title{The $A_\infty$-Structure of the Index Map}
\author{Oliver Braunling \& Michael Groechenig \& Jesse Wolfson}

\address{Freiburg Institute for Advanced Studies (FRIAS), University of Freiburg}
\email{oliver.braeunling@math.uni-freiburg.de}
\address{Department of Mathematics, Freie Universit\"{a}t Berlin}
\email{m.groechenig@fu-berlin.de}
\address{Department of Mathematics, University of California, Irvine}
\email{wolfson@uci.edu}

\begin{abstract}
    Let $F$ be a local field with residue field $k$. The classifying space of $\GL_n(F)$ comes canonically equipped with a map to the delooping of the $K$-theory space of $k$. Passing to loop spaces, such a map abstractly encodes a homotopy coherently associative map of $A_\infty$-spaces $\GL_n(F)\to K_k$. Using a generalized Waldhausen construction, we construct an explicit model built for the $A_\infty$-structure of this map, built from nested systems of lattices in $F^n$. More generally, we construct this model in the framework of Tate objects in exact categories, with finite dimensional vector spaces over local fields as a motivating example.
\end{abstract}

\thanks{O.B.\ was supported by DFG SFB/TR 45 ``Periods, moduli spaces and arithmetic of algebraic varieties'' and the Alexander von Humboldt Foundation. M.G.\ was supported by EPRSC Grant No.\ EP/G06170X/1. J.W.\ was partially supported by an NSF Graduate Research Fellowship under Grant No.\ DGE-0824162, by an NSF Research Training Group in the Mathematical Sciences under Grant No.\ DMS-0636646, and by an NSF Post-doctoral Research Fellowship under Grant No.\ DMS-1400349. He was a guest of K. Saito at IPMU while this paper was being completed. Our research was supported in part by NSF Grant No.\ DMS-1303100 and EPSRC Mathematics Platform grant EP/I019111/1.}


\maketitle

\section{Introduction}
Let $F$ be a local field with residue field $k$, e.g. $F=\Qb_p$ and $k=\Fb_p$, or $F=\Fb_p((t))$ and $k=\Fb_p$. Let $O\subset F$ be the ring of integers, $\mathfrak{m}\subset O$ the maximal ideal, and denote by $\Tor_{\mathfrak{m},f}(O)$ the category of finitely generated torsion $O$-modules. Let $S_{\bullet}$ denote Waldhausen's $S$-construction.  For any finite dimensional vector space $V$ over $F$, the authors constructed in \cite{Braunling2017} an ``index'' map, i.e. a map of spaces
\begin{equation*}
    B\GL(V)\to^{\Index} |S_\bullet(\Tor_{\mathfrak{m},f}(O))^\times|\to^\simeq BK_k
\end{equation*}
from the classifying space of $\GL(V)$, a group which we shall always tacitly view as equipped with the discrete topology, to Waldhausen's delooping of the $K$-theory space of $k$.\footnote{For $F$ such that $k$ is not a subfield of $F$, the existence of the map $|S_\bullet(\Tor_{\mathfrak{m},f}(O))^\times|\to BK_k$ relies on devissage.}

%%%%%%%%%%%%%%% CHANGES FOR REVISION 2 %%%%%%%%%%%%%%%%%%%%%%%%%%%%%%%%%
To sketch the bigger picture: for a local field $F$ with residue field $k$, Quillen's localization sequence gives a boundary
map%
\begin{equation}
\Omega K_{F}\longrightarrow K_{k}\text{,}\label{j_1}%
\end{equation}
where $K_{F}$ is the algebraic $K$-theory of the category of
finite-dimensional $F$-vector spaces. On the other hand, by a general procedure a finite-dimensional $F$-vector space $V$ can be written as an ind-pro limit of finite-dimensional $k$-vector spaces.  The \textquotedblleft index map\textquotedblright\ has the property that (the classifying space of) the group of automorphisms of $V$ as such an ind-pro limit can also be mapped to the $K$-theory $K_{k}$ of the residue field.
Restricting to those automorphisms which genuinely come from $F$-vector space
automorphisms, \cite{Braunling2017} shows that, suitably restricted to a common source, this
map agrees with the one coming from Equation \ref{j_1}.
%%%%%%%%%%%%%%%%%%%%%%%%%%%%%%%%%%%%%%%%%%%%%%%%%%%%%%%%%%%%%%%%%%%%%%%%

Let $\Vect_f$ denote the category of finite dimensional vector spaces. The index map encodes, after passing to loop spaces, a homotopy coherently associative map of loop spaces
\begin{equation*}
    \GL(V)\to^\simeq \Omega B\GL(V)\to\Omega|S_\bullet(\Vect_f(k))^\times|\to^\simeq K_k,
\end{equation*}
which in turn amounts to a coherent collection of homotopies
\begin{equation}\label{htpy}
    \Index(g_1)+\Index(g_2)\simeq\Index(g_1g_2).
\end{equation}
In applications, e.g. \cite{BGW:14}, one would like to be able to manipulate these homotopies in detail. The goal of the present paper is to construct a map of reduced Segal spaces
\begin{equation*}
    B_\bullet\GL(V)\to K_{S_\bullet(\Vect_f(k))}
\end{equation*}
whose geometric realization is the index map.\footnote{Here $B_\bullet G$ denotes the bar construction (or nerve) of the group $G$. This is a reduced Segal space with $|B_\bullet G|\simeq BG$.} Our main tool for this construction is a generalized Waldhausen construction, developed in Section \ref{sub:waldhausen}.  Our model for this construction follows from an analogy with index theory. Given an invertible element $f\in F^\times$ such that $f\cdot O\subset O$, the linear map $O\to^f O$ has finite dimensional cokernel, and the assignment $f\mapsto O/f\cdot O$ extends to a map of spaces
\begin{equation*}
    \GL_1(F)\to K_k.
\end{equation*}
To extend this to a full map of simplicial spaces (and to handle the case where $k$ is not a subfield of $F$, or when $\dim V>1$), we employ the framework of Tate objects in an exact category $\C$, as developed in \cite{BGW:2016MMJ}. Tate objects provide a setting for working with ``locally compact'' objects modeled on $\C$.  For example, a finite dimensional vector space over $\Qb_p$ is canonically a locally compact topological abelian group (with the $p$-adic topology), and also an elementary Tate object in the category $\Ab_{p,f}$ of finite abelian $p$-groups. A key advantage of working with Tate objects is that the category $\Tate(\C)$ of Tate objects in $\C$ is itself an exact category, and can be treated on the same footing as $\C$ (without requiring any topological constructions).

To define Tate objects, we rely on the notion of ``admissible Ind-objects''. Recall that an admissible Ind-object in $\C$ is a left exact presheaf $\widehat{X}$ of abelian groups on $\C$ such that $\widehat{X}$ can be written as the colimit of a filtering diagram $X\colon I\to \C$ in which all maps $X_i\to X_j$ are admissible monics. The category of admissible Ind-objects $\Ind(\C)$ is a full sub-category of the category $\Lex(\C)$ of all left exact presheaves of abelian groups, and it inherits an exact structure from $\Lex(\C)$ (cf. \cite[Section 3]{BGW:2016MMJ}). We define the category of admissible Pro-objects by $\Pro(\C):=\Ind(\C^{\op})^{\op}$. Since $\Pro(\C)$ is an exact category, we can consider the exact category $\Ind(\Pro(\C))$, and we define $\elTate(\C)$ to be the smallest full sub-category of $\Ind(\Pro(\C))$ which contains $\Ind(\C)$ and $\Pro(\C)$ and is closed under extensions.

The key feature of Tate objects is that they have ``lattices'', i.e. admissible sub-objects $L\subset V$ such that $L\in\Pro(\C)$ and $V/L\in\Ind(\C)$. For example, the ring of integers $\Zb_p\subset\Qb_p$ is canonically an object in $\Pro(\Ab_{p,f})$, and $\Qb_p/\Zb_p$ is a discrete abelian $p$-group, or equivalently, an object of $\Ind(\Ab_{p,f})$. In the above analogy with index theory, any Tate object $V$ can play the role of $F$, any lattice $L\subset V$ the role of $O$, and any automorphism $g\in\GL(V)$ the role of $f\in F^\times$. Following this analogy, coherent homotopies as in \eqref{htpy} should correspond to choices of nested systems of lattices in $V$. In the present paper, we make this precise by using a generalized Waldhausen construction to exhibit, for a Tate object $V$ in an idempotent complete exact category $\C$, a map of reduced Segal objects
\begin{equation}\label{segalmapInIntro}
    B_\bullet\GL(V)\to K_{S_\bullet(\C)}
\end{equation}
whose geometric realization is the index map. The present construction will be independent of our approach in \cite{Braunling2017}. In Subsection \ref{sec:compare}, we exhibit a homotopy between the geometric realization of \eqref{segalmapInIntro} and the ``index map'' of \cite{Braunling2017}.
\\
\paragraph*{{\bf Acknowledgements}}
We thank the editor Chuck Weibel and the anonymous referee for numerous helpful suggestions which have greatly improved the paper.

\section{Preliminaries}\label{sec:prelim}
Throughout this paper we work in the $\infty$-categories of spaces and spectra.  We take \cite{Lurie:bh} and \cite{Lurie:ha} as standard references for $\infty$-categories.
\\
\paragraph{\bf{Exact Categories and Tate Objects}}\label{par:et}
We follow the notation of \cite{Braunling2017} throughout. We consider \emph{exact categories} $\Cc$, i.e. additive categories equipped with a collection of distinguished kernel-cokernel pairs
\begin{equation*}
    X\into Y\onto Z
\end{equation*}
called \emph{exact sequences} which satisfy axioms modeled on the behavior of exact sequences of abelian groups or of projective modules. See B\"uhler \cite{MR2606234} for an excellent exposition. An exact category $\C$ is \emph{idempotent complete} if every idempotent splits, i.e. if for all $p\colon X\to X$ in $\C$ with $p^2=p$, there exists an isomorphism $X\cong Y\oplus Z$ which takes $p$ to $1_Y\oplus 0$. Fixing language, we refer to maps which arise as kernels, respecively cokernels, of exact sequences as \emph{admissible monics}, respectively \emph{admissible epics}.

Given an exact category $\Cc$, there are associated exact categories $\Ind(\Cc)$ and $\Pro(\Cc)$ of \emph{admissible Ind-objects} and \emph{admissible Pro-objects} and also exact categories $\elTate(\Cc)$ and $\Tate(\Cc)$ of \emph{elementary Tate objects} and \emph{Tate objects} in $\Cc$. We quickly recall the definitions here, and refer the reader to \cite{BGW:2016MMJ} for full details.

Denote by $\Lex(\C)$ the abelian category of left exact presheaves of abelian groups on $\C$. The Yoneda embedding allows us to view $\C$ as a fully exact sub-category of $\Lex(\C)$ which is closed under extensions (cf. e.g. \cite[Appendix A]{Kel:90}).
\begin{definition}
    Let $\C$ be an exact category. An \emph{admissible Ind-object} in $\C$ is an object $\widehat{X}\in\Lex(\C)$ such that $\widehat{X}$ is the colimit (in $\Lex(\C)$) of a filtering diagram $X\colon I\to \C$ in which all maps $X_i\to X_j$ are admissible monics in $\C$. Define the category of admissible Ind-objects $\Ind(\C)$ as a full sub-category of $\Lex(\C)$.  Define the category of \emph{admissible Pro-objects} $\Pro(\C)$ by $\Pro(\C):=\Ind(\C^{\op})^{\op}$.
\end{definition}
Following Keller \cite[Appendix A]{Kel:90}, we show in \cite[Theorem 3.7]{BGW:2016MMJ} that $\Ind(\C)$ is closed under extensions in $\Lex(\C)$, and thus has a canonical structure as an exact category.

\begin{rmk}
    Unpacking the definitions, one can also realize $\Pro(\C)$ as a localization of the category $\Inv(\C)$ of co-filtering systems of admissible epimorphisms, where one localizes at all morphisms of diagrams which are invertible on a cofinal subdiagram. Equivalently, one localizes at all morphisms which become invertible under the evaluation map $\Inv(\C)\to\Lex(\C^{\op})^{\op}$.
\end{rmk}

\begin{definition}
    Let $\C$ be an exact category. Define the category of \emph{elementary Tate objects} $\elTate(\C)$ to be the smallest full sub-category of $\Ind(\Pro(\C))$ which contains $\Ind(\C)$ and $\Pro(\C)$ and which is closed under extensions. Define the category of \emph{Tate objects} $\Tate(\C)$ to be the idempotent completion of $\elTate(\C)$.
\end{definition}
By \cite[Theorem 5.6]{BGW:2016MMJ}, the category of elementary Tate objects is well-defined, and thus inherits a canonical exact structure from $\Ind(\Pro(\C))$.

\begin{example}
    Let $\Ab_{p,f}$ be the category of finitely generated abelian $p$-groups. There exists an exact functor
    \begin{equation*}
        \Vect_{f}(\Qb_p)\to \elTate(\Ab_{p,f})
    \end{equation*}
    from the category of finite dimensional vector spaces over $\Qb_p$ to the category of elementary Tate objects in $\Ab_{p,f}$.
\end{example}

For the present, we will need the following.
\begin{definition}\label{def:sato}
    Let $V$ be an elementary Tate object in $\C$.
    \begin{enumerate}
        \item A \emph{lattice} $L\hookrightarrow V$ is an admissible sub-object, with $L\in\Prok(\C)$ and the cokernel $V/L\in\Indk(\C)$.
        \item The \emph{Sato Grassmannian} $\Gr(V)$ is the partially ordered set of lattices in $V$, where $L_0\leq L_1$ if there exists a commuting triangle of admissible monics
            \begin{equation*}
                \xymatrix{
                L_0 \ar@{^{(}->}[r] \ar@{^{(}->}[rd] & L_1 \ar@{^{(}->}[d] \\
                & V
                }
            \end{equation*}
    \end{enumerate}
\end{definition}

Lattices and the Sato Grassmannian play a key role in our study of Tate objects. We view Assertion (c) in the theorem below as the main result of \cite{BGW:2016MMJ}.
\begin{theorem}\label{thm:sato_filtered}(\cite[Prop. 6.6, Thm. 6.7]{BGW:2016MMJ})
    Let $\C$ be an exact category.
    \begin{enumerate}
        \item[(a)] Every elementary Tate object in $\C$ has a lattice.
        \item[(b)] The quotient of a lattice by a sub-lattice is an object of $\C$.
        \item[(c)] If $\C$ is idempotent complete, and $L_0\hookrightarrow V$ and $L_1\hookrightarrow V$ are two lattices in an elementary Tate object $V$, then there exists a lattice $N\hookrightarrow V$ with $L_0,L_1\le N$ in $\Gr(V)$. Similarly, $L_0$ and $L_1$ have a common sub-lattice $M \le L_0,L_1$.
    \end{enumerate}
\end{theorem}

\paragraph{\bf{Algebraic K-Theory}}\label{par:k}
Following Quillen \cite{MR0338129}, one associates to every exact category $\Cc$ its $K$-theory space $K_{\C}$. The space $K_{\C}$ is an infinite loop space which serves as a universal target for \emph{additive} invariants of $\Cc$. In \cite{MR0802796}, Waldhausen gave an alternate construction of $K_{\C}$, and proved his fundamental ``Additivity Theorem''. Waldhausen's treatment of algebraic $K$-theory hinges on two simplicial exact categories, denoted by $S_\bullet(\C)$, and $S^r_{\bullet}(f)$ respectively, where $\C$ is an exact category and $f\colon \C \to \D$ is an exact functor. The simplicial object $S_\bullet(\C)$ associates to every finite non-empty totally ordered set $[k]$ the exact category $S_k(\C)$, which consists of functors $[k] \to \C$, sending every arrow in $[k]$ to an admissible monic. Likewise, the simplicial object $S_\bullet^r(f)$ associates to $[k]$ the exact category $S_k(\C)$ consisting of functors $[k]\to\D$, sending every arrow in $[k]$ to an admissible monic in $\D$ with cokernel in $\C$. Given a category $\Cc$, denote by $\Cc^\times$ the groupoid of all isomorphism in $\Cc$. With this notation, Waldhausen's definition can be given as
\begin{equation*}
    K_{\C}:=\Omega|S_\bullet(\C)^\times|.
\end{equation*}
See \cite[Section 2]{Braunling2017} for a discussion of Waldhausen's approach to $K$-theory tailored to the present setting. As discussed there, the fundamental property of algebraic $K$-theory is the following ``Additivity Theorem''. The results of this paper and \cite{Braunling2017} can be seen as consequences of the Additivity Theorem combined with Theorem \ref{thm:sato_filtered}.
\begin{theorem}[Waldhausen's Additivity Theorem]\label{thm:wald_add1}(\cite[Theorem 1.4.2, Proposition 1.3.2(4)]{MR0802796})
    Let $F_1\hookrightarrow F_2\twoheadrightarrow F_3$ be an exact sequence of functors $\C_1\to\C_2$. Then the map
    \begin{align*}
        |S_\bullet F_2|\colon |S_\bullet(\C_1)^\times|&\to|S_\bullet(\C_2)^\times|\intertext{is naturally homotopic to}
        |S_\bullet F_1\oplus S_\bullet F_3|\colon |S_\bullet(\C_1)^\times|&\to|S_\bullet(\C_2)^\times|.
    \end{align*}
\end{theorem}

Several equivalent reformulations exist. We will need the following.
\begin{definition}[Waldhausen]
    Let $\D$ be an exact category, and let $\C_1$ and $\C_2$ be full sub-categories of $\D$ which are closed under extensions.
    Define $\E(\C_1,\D,\C_2)$ to be the full sub-category of $\E\D$ consisting of the exact sequences $X_1\hookrightarrow
    Y\twoheadrightarrow X_2$ with $X_i\in\C_i$.
\end{definition}
Note that, because $\C_1$ and $\C_2$ are closed under extensions in $\D$, $\E(\C_1,\D,\C_2)$ is closed under extensions in
$\E\D$; in particular, it is an exact category.

\begin{theorem}\label{thm:additivity}
    Let $\A \to^i \B\to^p \C$ be a composable pair of exact functors such that $i$ is fully faithful and induces an equivalence with the full sub-category of $\B$ annihilated by $p$. Moreover, assume that $p$ has a left adjoint
    $$s\colon \C \to \B,$$
    such that $ps\cong 1_{\C}$ and such that, for every object $Y\in\B$, the co-unit $sp(Y)\to Y$ is an admissible monic with cokernel in $\A$. Then, the map
    $$i \times s\colon K_{\A} \times K_{\C} \xrightarrow{\simeq} K_{\B},$$
    is an equivalence of spaces.
\end{theorem}
While this theorem is, without doubt, well-known, we have a chosen a less conventional statement which is convenient for our applications. Therefore, we now give a proof.
\begin{proof}
    We have a well-defined map of spaces $i \times s\colon K_{\A} \times K_{\C} \rightarrow  K_{\B}$.  By the Whitehead lemma it suffices to show that it establishes an equivalence on all homotopy groups.

    The admissible monic of functors
    $$sp\hookrightarrow 1_{\B}\colon \B \to \B,$$ given by the co-unit of the adjunction $(p,s)$, extends to a short exact sequence
    $$sp\hookrightarrow 1_{\B} \twoheadrightarrow f\colon \B \to \B.$$
    By construction, $pf = 0$, therefore $f$ can be expressed as $ir$, where $r\colon \B \to \A$ is an exact functor. By the Additivity Theorem (Theorem \ref{thm:wald_add1}), we have
    $$\pi_i(K(ir)\oplus K(sp)) = \pi_i(K(1_{\B})).$$
    Moreover, the relations $ps = 1_{\C}$ and $ri = 1_{\A}$ imply that we also have
    $$\pi_i(K_{\B}) \cong \pi_i(K_{\A}) \times\pi_i(K_{\C}).$$
    The Whitehead lemma concludes the proof.
\end{proof}

\paragraph{\bf{Segal Objects}}
In \cite{MR0353298}, Segal introduced a definition which, in the hands of May and Thomason \cite{MR0508885}, \cite{MR0534053}, Rezk \cite{MR1804411}, Lurie \cite{Lurie:ha} and many others, has become fundamental to the study of $A_\infty$-objects (a.k.a. $E_1$-objects or homotopy coherent associative monoids) in a homotopical setting.
\begin{definition}
    Let ${\C}$ be an $\infty$-category with finite products. For each $n$, consider the collection of maps
    \begin{equation*}
      \{[1]=\{0<1\}\to^\cong\{i-1<i\}\subset[n]\}_{i=1}^n
    \end{equation*}
    A \emph{Segal object} in ${\C}$ is a simplicial object $X_\bullet\in\Fun(\Ddelta^{\op},{\C})$ such that, for $n\ge 2$, the map
    \begin{equation*}
        X_n\to \underbrace{X_1\times_{X_0}\cdots\times_{X_0} X_1}_n
    \end{equation*}
    induced by the above collection is an equivalence. A \emph{reduced Segal object} $X_\bullet$ is a Segal object with $X_0\simeq\ast$. Segal objects form a full sub-category of simplicial objects in $\C$.
\end{definition}

For a basic example, the bar construction associates to a group $G$ a simplicial space $B_\bullet G$ with $n$-simplices the discrete space $G^n$.  A standard exercise shows that $B_\bullet G$ is a reduced Segal space, and the Segal structure is just a rewriting of the group law.  For a richer example, given an exact category $\C$, we can consider the simplicial exact category $S_\bullet\C$ given by Waldhausen's $S_\bullet$-construction. Waldhausen's Additivity Theorem (Theorem \ref{thm:wald_add1}) implies that the simplicial space $K_{S_\bullet \C}$ obtained by taking the $K$-theory space of each category of $n$-simplices is a reduced Segal object in the $\infty$-category of spaces. The Segal space structure encodes the homotopy coherent addition of elements in $K_{\C}$.
\\
\paragraph{\bf{The Index Map}}
We now recall the index map. For $n\ge 0$, denote by $[n]$ the partially ordered set $\{0<\ldots<n\}$ viewed as a category, and, for a category $\C$, denote by $\Fun([n],\C)$ the category of functors from $[n]$ to $\C$.

\begin{definition}
    Let $\C$ be an exact category. Define the \emph{Sato complex} $\Gr_\bullet^\le(\C)$ to be the simplicial diagram of exact
    categories with
    \begin{enumerate}
        \item $n$-simplices $\Gr_n^\le(\C)$ given by the full sub-category of $\Fun([n+1],\elTate(\C))$ consisting of sequences of admissible monics
            \begin{equation*}
                L_0\hookrightarrow\cdots\hookrightarrow L_n\hookrightarrow V
            \end{equation*}
            where, for all $i$, $L_i\hookrightarrow V$ is the inclusion of a lattice,
        \item face maps are given by the functors
            \begin{equation*}
                d_i(L_0\hookrightarrow\cdots\hookrightarrow L_n\hookrightarrow V):=(L_0\hookrightarrow\cdots\hookrightarrow
                L_{i-1}\hookrightarrow L_{i+1}\hookrightarrow\cdots\hookrightarrow L_n\hookrightarrow V),
            \end{equation*}
        \item and degeneracy maps are given by the functors
            \begin{equation*}
                s_i(L_0\hookrightarrow\cdots\hookrightarrow L_n\hookrightarrow V):=(L_0\hookrightarrow\cdots\hookrightarrow
                L_i\hookrightarrow L_i\hookrightarrow\cdots\hookrightarrow L_n\hookrightarrow V).
            \end{equation*}
    \end{enumerate}
\end{definition}

The simplicial object $\Gr^{\leq}_{\bullet}(\C)$ allows us to introduce the index map.

\begin{definition}
    Let $\C$ be an exact category. The \emph{categorical index map} is the span of simplicial maps
    \begin{equation}\label{catindex}
        \elTate(\C)\longleftarrow\Gr_\bullet^\le(\C)\to^\Index S_\bullet(\C),
    \end{equation}
    where the left-facing arrow is given on $n$-simplices by the assignment
    \begin{equation*}
        (L_0\hookrightarrow\cdots\hookrightarrow L_n\hookrightarrow V)\mapsto V,
    \end{equation*}
    and $\Index$ is given on $n$-simplices by the assignment
    \begin{equation*}
        (L_0\hookrightarrow\cdots\hookrightarrow L_n\hookrightarrow V)\mapsto(L_1/L_0\hookrightarrow\cdots\hookrightarrow
        L_n/L_0).
    \end{equation*}
\end{definition}

Recall the following.
\begin{proposition}\label{prop:gr}\cite[Proposition 3.3]{Braunling2017}\mbox{}
    Let $\C$ be an idempotent complete exact category. Then the map from $\Gr_\bullet^\le(\C)\to\elTate(\C)$ of \eqref{catindex} induces an equivalence
    \begin{equation}\label{satoequiv}
        |\Gr_\bullet^\le(\C)^\times|\to^\simeq|\elTate(\C)^\times|.
    \end{equation}
\end{proposition}

\begin{rmk}
    The proposition follows from the fact that if $\C$ is idempotent complete, then the Sato Grassmannian $\Gr(V)$ of every element Tate object is a directed and co-directed poset (\cite[Theorem 6.7]{BGW:2016MMJ}). The nerve of this poset is therefore contractible, and these geometric realization of these nerves are the fibers of the map \eqref{satoequiv}.
\end{rmk}

Following the proposition, we obtain the $K$-theoretic index map by restricting the categorical index map \eqref{catindex} to the groupoids of all isomorphisms, geometrically realizing, and picking a homotopy inverse to \eqref{satoequiv} to obtain the map
\begin{equation}\label{kindex}
    \Index\colon |\elTate(\C)^\times|\to^\simeq |\Gr_\bullet^\le(\C)^\times|\to |S_\bullet(\C)^\times|=:BK_{\C}.
\end{equation}
Our goal is to construct an explicit map of Segal objects $B_\bullet\Aut(V)\to K_{S_\bullet(\C)}$, for any elementary Tate object $V$, whose geometric realization is equivalent to the restriction of \eqref{kindex} along the map $|*//\Aut(V)|\to |\elTate(\C)^\times|$.\footnote{Here $*//G$ denotes the one object groupoid with automorphisms $G$, and the map $*//\Aut(V)\to\elTate(\C)^\times$ is given on objects by $*\mapsto V$ and is the identity map on automorphisms.}

\section{The $A_\infty$-Structure of the Index Map}\label{e1}
\subsection{A Generalized Waldhausen Construction}\label{sub:waldhausen}
Let $\C$ be an exact category, and $f\colon \C\to\D$ an exact functor.  Waldhausen's approach to algebraic $K$-theory \cite{MR0802796} hinges on the simplicial exact categories $S_\bullet(\C)$ and $S^r_{\bullet}(f)$ recalled above. We now extend the functors
$$S_\bullet(\C),\;S^r_{\bullet}(f) \colon \Delta^{\op} \to \Excat$$
from the ordinal category, i.e. the category of finite non-empty linearly ordered sets, to the category of filtered finite partially ordered sets. We refer to the resulting functors as the ``generalized Waldhausen construction''.  In Subsection \ref{sub:revisit} we then use the generalized Waldhausen construction to give a treatment of the $A_\infty$-structure of the index map.

\subsubsection{Partially Ordered Sets and Related Structures}\label{subsub:posets}
The current paragraph contains several definitions of combinatorial nature.
\begin{definition}\label{defi:graph}
Let $I$ be a partially ordered set. We denote by $\Gamma(I)$ the directed graph given by the set underlying $I$ as set of vertices, and intervals $a < b$ as edges. We denote the set of directed edges of $\Gamma(I)$ by $E(I)$.
\end{definition}

\begin{example}
For the ordinal $[2]$ we obtain

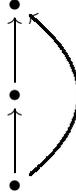
\[
\xymatrix{
\bullet \ar@{<-}[d]  \ar@{<-}@/^2.0pc/[dd] \\
\bullet \ar@{<-}[d]\\
\bullet
}
\]
for the oriented graph $\Gamma([2])$. While this graph is more traditionally drawn as the boundary of a 2-simplex, the present depiction is chosen to highlight the maximal tree.
\end{example}

We will work with finite, filtered, partially ordered sets with \emph{basepoints} (which are chosen to be minimal elements).
\begin{definition}\label{defi:based_posets}
A \emph{based, finite, filtered, partially ordered set} is a pair $(I;x_0,\dots,x_k)$, where $I$ is a \emph{finite} partially ordered set with a final element, and $(x_0,\dots,x_k)$ is a tuple of  \emph{minimal elements} in $I$.\footnote{It is important to note that the basepoints are not assumed to be pairwise distinct.} A morphism of based partially ordered sets is a map of pairs $$(f,\sigma)\colon(I;x_0,\dots,x_k) \to (I';y_0,\dots,y_{m}),$$ where $f\colon I\to I'$ is a map of partially ordered sets, $\sigma\colon [m] \to [k]$ is a map of finite ordinals, and where $f({x_i}) = y_{\sigma(i)}$. The category of based, finite, filtered, partially ordered sets will be denoted by $\bSet$.
\end{definition}
The assumption of \emph{finiteness} is crucial for the inductive proofs that are given later, but could eventually be relaxed.

Some arguments require us to choose a maximal tree in $\Gamma(I)$ with good properties.
\begin{definition}\label{defi:admissible_tree}
Let $\Gamma$ be an oriented graph. A maximal tree $T \subset \Gamma$ is said to be \emph{admissible}, if for every pair of vertices $(x,y)$, there exists a vertex $z$, and a unique oriented path from $x$ to $z$, respectively $y$ to $z$ within $T$.
\end{definition}

The following examples help to clarify this definition.

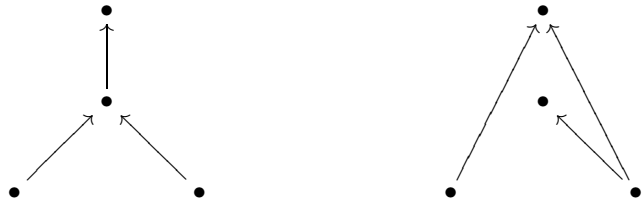
\begin{example}
\[
\xymatrix{
& \bullet &  & & & & \bullet & \\
& \bullet \ar[u] & & & & & \bullet  & \\
\bullet \ar[ur] & & \bullet \ar[ul] & & & \bullet \ar[uur] & & \bullet \ar[ul] \ar[uul]
}
\]
The tree on the left is admissible, while the one on the right is not (there is no common vertex that receives an oriented path from the two upper vertices).
\end{example}

\begin{example}
    Let $I$ be a finite, filtered, partially ordered set. An admissible tree $T\subset\Gamma(I)$ always exists. Indeed, let $m\in I$ denote the final element. Then the tree $T$ given by the union of all edges $(x,m)$ for $x \in I$, is admissible.
\end{example}

The definition below introduces the concept of a \emph{framing} of a based partially ordered set.

\begin{definition}\label{defi:framed}
A \emph{framed partially ordered set} is a triple $(I,E(T),x_0,\dots,x_k)$, where $E(T) \subset E(I)$ is the set of edges of an admissible maximal tree, and $(I;x_0,\dots,x_k)$ is a based, finite, filtered, partially ordered set. The category of framed, partially ordered sets $\fPoset$ is the category with framed, partially ordered sets as objects, and morphisms
$$\phi\colon(I,E(T),x_0) \to (I',E(T'),x'_0),$$
where $\phi\colon I \to I'$ is a map of partially ordered sets, mapping the basepoints \emph{bijectively} onto each other, and satisfying $\phi(T) \subset \phi(T')$. We denote by
$$\phi_{\sharp}\colon E(T) \to E(T')_+ = E(T) \cup \{\star\}$$
the map which sends $e \in E(T)$ either to its image $\phi(e) \in E(T')$, or, if $\phi(e)$ consists of a single point, to the basepoint $\star$.
\end{definition}

\subsubsection{Pairs of Exact Categories and Diagrams}\label{subsub:diagrams}
We define the generalized Waldhausen construction in the context of extension closed sub-categories of exact categories.
\begin{definition}\label{defi:pair}
    We denote by $\Expair$ the $2$-category of pairs of exact categories $\C \subset \D$, such that $\C$ is an \emph{extension-closed} sub-category of $\D$. Objects in this category will also be referred to using the notation $(\D,\C)$.
\end{definition}

For every partially ordered set $I$ we have an associated category. For notational convenience, we will not distinguish between these.

\begin{definition}\label{defi:admissible_diagrams}
Let $(\D,\C)\in \Expair$ be a pair of exact categories. Let $I$ be a partially ordered set. An \emph{admissible $I$-diagram} in $(\D,\C)$ is a functor $I \to \D$, sending each arrow in $I$ to an admissible monic in $\D$ with cokernel an object of $\C$. We denote the exact category of such functors by $\Fun_{\C}(I,\D)$.
\end{definition}
The following example serves as a motivation for this definition.
\begin{example}
    We observe that $\Fun_{\C}([n],\D) = S_n^r(\C\subset\D)$ (see Paragraph \ref{par:k}).
\end{example}

In Definition \ref{defi:framed} we introduced the concept of framed, partially ordered sets. Recall the map $\phi_{\sharp}\colon E(T) \to E(T')_+$. By abuse of notation we will also use the symbol $\phi_{\sharp}$ to denote the unique map of pointed sets
$$E(T)_+ \to E(T')_+.$$
Note that, for every object $X$ in a pointed $\infty$-category $\Cc$ with finite coproducts, we have a natural functor
$$\coprod_{?}X\colon(\Set^{\mathrm{fin}}_{\ast})^{\op} \to \Cc.$$
An inductive argument allows us to establish the following lemma. The choice of a maximal tree $T \subset \Gamma(I)$ should be understood as an analogous to choosing a basis for a vector space.

\begin{lemma}\label{lemma:inductive}
Let $(I;E(T),x_0,\dots, x_k)$ be a framed, partially ordered set. We denote by $T\subset \Gamma(I)$ an admissible maximal tree of $\Gamma(I)$. Then there exists an equivalence
$$\phi(T)\colon { }K_{\Fun_{\C}(I,\D)} \cong { }K_{\D} \times { }K_{\C}^{\times E(T)}.$$
Moreover, this equivalence can be seen as a natural equivalence of functors
$${ }K_{\Fun_{-}(-,-)} \simeq { }K_- \times { }K_-^{\times E(-)} \colon\Expair \times (\fPoset)^{\op} \to \Spaces.$$
\end{lemma}

Although the lemma is stated for a framed partially ordered set with basepoints $x_0,\dots, x_k$, we actually only need the zeroth basepoint $x_0$. An inspection of the proof below shows that all the other basepoints could be discarded.

\begin{proof}[Proof of Lemma \ref{lemma:inductive}]
For every $e = (y_i \leq y_{i+1}) \in E(T)$ we denote by $X_e = F(y_{i+1})/F(y_{i})$. We have an exact functor
$${\Fun_{\C}(I,\D)} \to \D \times \C^{E(T)}$$
which sends $F\colon I \to \D$ to $(F(x_0),(X_e)_{e \in E(T)})$. This map defines a natural transformation between the functors
$$\Fun_-(-,-),\;(-) \times (-)^{E(-)}\colon \Expair \times (\fPoset)^{\op} \to \Excat.$$
Applying the functor ${ }K_{-}\colon \Excat \to \Spaces$, we obtain the natural transformation $\phi(T).$ It remains to show that $\phi(T)$ is an equivalence for each triple $(I,\D,\C)$. We will use induction on the cardinality of $I$ to show this. As a warmup, we begin with the case that $I$ is a totally ordered set. Without loss of generality we may identify it with $\{0 < \dots < n\}$. Moreover, in the totally ordered case, there is only  one possible choice for the framing $(T,x_0)$. The induction is anchored to the case $n = 0$, i.e. the case of the singleton set, which is evidently true.

Assume that $\phi(T)$ has been shown to be an equivalence for totally ordered sets of cardinality $< n$. We denote by $I'$ the framed partially ordered set defined by the subset $\{0 < \dots < n-1\}$. The restriction functor $\Fun_{\C}(I,\D)\to \Fun_{\C}(I',\D)$ sits in a short exact sequence of exact categories
$$\C \hookrightarrow \Fun_{\C}(I,\D)\twoheadrightarrow \Fun_{\C}(I',\D),$$
where we send $X \in \C$ to $(0\hookrightarrow\cdots\hookrightarrow 0\hookrightarrow X) \in \Fun_{\C}(I,\D)$. We also have a splitting, given by
$$\Fun_{\C}(I',\D)\to \Fun_{\C}(I,\D),$$
which sends $(Y_0 \hookrightarrow \cdots \hookrightarrow Y_{n-1})$ to $(Y_0 \hookrightarrow \cdots \hookrightarrow Y_{n-1} \hookrightarrow Y_{n-1})$. By means of the Additivity Theorem \ref{thm:additivity}, we conclude
$${ }K_{\Fun_{\C}(I,\D)} \cong { }K_{\Fun_{\C}(I',\D)} \times { }K_{\C}.$$
Applying the inductive hypothesis to $\Fun_{\C}(I',\D)$, we conclude the assertion for totally ordered sets.

The proof for general $I$ also works by induction on the number of elements. If $I$ is not totally ordered, but of cardinality $n + 1$, we may decompose our framed partially ordered set
$$(I,T) = (I',T') \cup (I'',T''),$$
where $I''$ is totally ordered, $I' \cap I'' = \{\max I''\}$, and $x_0 \in I'$. Consider for example the graph:
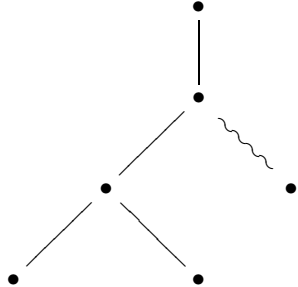
\[
\xymatrix{
& & \bullet \ar@{-}[d] & & \\
& & \bullet \ar@{-}[ld] \ar@{~}[rd] & & \\
& \bullet \ar@{-}[ld] \ar@{-}[rd] & & \bullet & \\
\bullet & & \bullet &
}
\]
where edges belonging to $I''$ have been drawn as squiggly lines.

There exists a positive integer $1 \leq k \leq n$, such that $I'' \cong \{0<\dots < k\}$. The restriction functor from $I$-diagrams to $I'$-diagrams belongs to a short exact sequence of exact categories
$$\Fun_{\C}(I''\setminus \{\max I''\},\C) \hookrightarrow \Fun_{\C}(I,\D) \twoheadrightarrow \Fun_{\C}(I',\D),$$
where the left hand side is seen as the exact category of morphisms $(Y_0 \hookrightarrow Y_1 \hookrightarrow \cdots \hookrightarrow Y_{k-1})$, which extends to an $I$-diagram by sending the object $Y_{k-1}$ to every vertex in $I'$. This short exact sequence is split by the functor
$$\Fun_{\C}(I',\D) \to \Fun_{\C}(I,\D),$$
which extends an $I'$-diagram to an $I$-diagram, by sending each vertex $y$ of $I''$ to the object $\max I'' \in I' \cap I''$ (with the identity morphisms as admissible epimorphisms between them). The Addivitity Theorem \ref{thm:additivity} yields
$${ }K_{\Fun_{\C}(I,\D)} \cong { }K_{\Fun_{\C}(I',\D)} \times { }K_{S_k(\C)}.$$
Using the induction hypothesis, we see that the first component is equivalent to ${ }K_{\D} \times { }K_{\C}^{\times E(T')}$, and the second component to ${ }K_{\C}^{\times E(T'')}$. This proves the assertion.
\end{proof}

\subsubsection{The Index Space}\label{subsub:indexspace}

Let $(I;x_0,\dots,x_k)$ be a based, finite, filtered, partially ordered set (Definition \ref{defi:based_posets}). Together with a pair of exact categories $\C \subset \D$, such that $\C$ is extension-closed in $\D$, we define the \emph{index space}, which is the recipient of a map from ${ }K_{\Fun_{\C}(I,\D)}$. It can be thought of as measuring the difference between the basepoints.

\begin{definition}\label{defi:Delta}
\begin{itemize}
\item[(a)] For a based, finite, filtered, partially ordered set $(I;x_0,\dots,x_,)$ we denote by $I^{\Delta}$ the partially ordered set obtained by identifying the basepoints. Cofunctoriality of $\Fun_{\C}(-,\C)$ yields a forgetful functor $\Fun_{\C}(I^{\Delta},\D) \to \Fun_{\C}(I,\D).$
\item[(b)] For an exact category $\D$, let $\Kk_{\D}$ be the connective $K$-theory spectrum. We denote by $\Idx_{\C,I}\D$ the space underlying (i.e. $\Omega^{\infty}$ of) the cofibre of the morphism\footnote{The long exact sequence of homotopy groups implies that this cofibre is again a connective spectrum.}  $${ }\Kk_{\Fun_{\C}(I^{\Delta},\D)} \to { }\Kk_{\Fun_{\C}(I,\D)}.$$
By functoriality of cofibres, this gives rise to a functor
$$\Idx\colon \Expair \times (\bSet)^{\op} \to \Spaces.$$
We refer to $\Idx_{\C,I}\D$ as the \emph{Index-space} of the pair $(\D,\C)$ relative to $(I;x_0,\dots,x_k)$.
\item[(c)] We refer to the map of spaces
\[
    |\Fun_{\C}(I,\D)^\times|\to K_{\Fun_C(I,D)}\to \Idx_{C,I}(D)
\]
as the {\em pre-index map} of the pair $(\D,\C)$ relative to $(I;x_0,\ldots,x_k)$.
\end{itemize}
\end{definition}

The index space is to a large extent independent of $I$, as guaranteed by its functorial nature in Definition \ref{defi:Delta}(b). We record this observation in the next two results. In Proposition \ref{prop:rigidity} we will further refine this statement.

\begin{lemma}\label{lemma:Idx_can}
Let $\C \hookrightarrow \D$ be an extension-closed exact sub-category of an exact category $\D$. We consider an injective morphism of finite, based, filtered, partially ordered sets, in the sense of Definition \ref{defi:based_posets},
$$(I;x_0,\dots,x_k) \to (I';y_0,\dots,y_k),$$
which induces a bijection of basepoints (i.e. on basepoints, it corresponds to the identity map $[k] \to [k]$). Then, the induced morphism of index spaces
$$\Idx_{\C,I}\D \to \Idx_{\C,I'}\D$$
is an equivalence.
\end{lemma}

\begin{proof}
By virtue of Lemma \ref{lemma:inductive}, the choice of an admissible maximal tree $T$ in $I$ induces an equivalence of $K$-theory spaces
$${ }K_{\Fun_{\C}(I,\D)} \cong { }K_{\D} \times { }K_{\C}^{\times E(T)}.$$
Recall from Definition \ref{defi:Delta} that $I^{\Delta}$ denotes the finite, based, filtered, partially ordered set obtained by identifying all basepoints. We can choose $T$ in a way, such that its image $T^{\Delta}$ in $I^{\Delta}$ is also an admissible tree. For instance, we could take the tree given by the edges $(x,m)$, where $m = \max I$, and $x$ runs through the elements of $I \setminus \{m\}$. We denote by $e_i$ the (unique) edge of $T$ which contains $x_i$. By construction, the edges $e_i$ map to the same edge in $T^{\Delta}$, and we denote this edge by $e$. We can apply the functoriality of Lemma \ref{lemma:inductive} to obtain the commutative square of connective $K$-theory spectra
\[
\xymatrix{
{ }\Kk_{\Fun_{\C}(I^{\Delta},\D)} \ar[r] \ar[d]_{\cong} & { }\Kk_{\Fun_{\C}(I,\D)} \ar[d]^{\cong} \\
{ }\Kk_{\D} \oplus { }\Kk_{\C}^{\oplus E(T^{\Delta})} \ar[r]^{\alpha} & { }\Kk_{\D} \oplus { }\Kk_{\C}^{\oplus E(T)},
}
\]
where the morphism $\alpha$ is given by the identity $1_{{ }\Kk_{\C}}$ for edges in $E(T) \setminus \{e_0,\dots,e_k\}$, and given by the diagonal map
$$\Delta_{{ }\Kk_{\C}}\colon{ }\Kk_{\C} \to { }\Kk_{\C}^{\oplus(k+1)},$$
for the component $e$. In particular, we see that $\mathrm{cofib}(\alpha) \cong \mathrm{cofib}(\Delta_{{ }K_{\C}})$.

The same analysis applies to $I'$. Because we can choose an admissible maximal tree $T$ in $I$ which extends to an admissible maximal tree $T'$ in $I'$, we see that $\mathrm{cofib}({ }\Kk_{\Fun_{\C}(I^{\Delta},\D)} \to { }\Kk_{\Fun_{\C}(I,\D)})$ is equivalent to
$$\mathrm{cofib}(\Delta_{{ }\Kk_{\C}}\colon { }\Kk_{\C} \to { }\Kk_{\C}^{k+1}) \cong \mathrm{cofib}({ }\Kk_{\Fun^{\Delta}_{\C}(I',\D)} \to { }\Kk_{\Fun_{\C}(I',\D)}).$$
The restriction functor $\Idx_{\C,I}\D \to \Idx_{\C,I'}\D$ is defined independently of any choices. The admissible maximal trees $T$ and $T'$ only play a role in verifying that this map is an equivalence. We therefore see that we have a canonical equivalence between $\Idx_{\C,I}\D$ and $\Idx_{\C,I'}\D$.
\end{proof}

\begin{definition}\label{def:Bk}
For every positive integer $k$ we have an object $B[k]=(B[k];b_0,\dots,b_k) \in \bSet$, given by the set of non-empty intervals in the ordinal $[k]$. An interval is understood to be a subset $J \subset [k]$ with the property that $x \leq y \leq z$ and $x,z \in J$ implies that $y \in J$. The basepoints $(b_i)_{i=0,\dots,k}$ are given by the singletons $\{i\}$.
\end{definition}
We have drawn the filtered partially ordered set $B[2]$ below.
\[
\xymatrix{
& & \bullet \ar@{-}[ld]\ar@{-}[rd] & & \\
& \bullet \ar@{-}[ld]\ar@{-}[rd] & & \bullet \ar@{-}[ld]\ar@{-}[rd] & \\
\bullet & & \bullet &  &\bullet
}
\]
\begin{definition}
For an arbitrary $I = (I,x_0,\dots,x_k)$ in $\bSet$, we denote by $I^B = (I^B;x_0,\dots,x_k)$ the based, finite, filtered, partially ordered set given by $I \cup B[k]$, where we identify the basepoints $b_i = x_i$ and where we extend the inductive ordering of $I$ to $I^B$ by demanding $x \leq y$, for all $x \in B[k]$ and $y \in I \setminus \{x_0,\dots,x_k\}$. To summarize the previous construction, we obtain $I^B$ from $I$ by gluing on a copy of $B[k]$ to $I$, with all new elements being $\leq$ than elements in $I$. This process is functorial in $I$, we denote the resulting functor by
$$(-)^B\colon \bSet \to \bSet.$$
The inclusion $I \subset I^B$ gives rise to a natural transformation of functors
$$1_{\bSet} \Rightarrow (-)^B.$$
\end{definition}

The category $\bSet$ satisfies the following property: for two objects $(I;x_0,\dots,x_k)$ and $(I';y_0,\dots,y_k)$ we can find an $(I'',z_1,\dots,z_k)$, containing sub-objects isomorphic to $I$ and $I'$ (respecting basepoints). Combining this observation with the lemma proven above, we obtain a complete description of index spaces.

\begin{corollary}\label{cor:pre-index}
Let $(I,x_0,\dots,x_k)$ be a based, finite, filtered, partially ordered set with pairwise distinct basepoints. Then, the index space of the pair $(\D,\C)$ is equivalent to
$${ }K_{S_k(\C)} \cong { }K_{\C}^{\times k},$$
This equivalence is functorial in the pair $\C \subset \D$, where $\C$ is extension-closed in $\D$, and it is contravariantly functorial in the based filtered partially ordered set $I$. Moreover, if $M_{\bullet}$ is a simplicial object in $\bSet$ such that, for every non-negative integer $k$, $M_k$ has $k+1$ pairwise distinct basepoints, then we have an equivalence of simplicial spaces
$$\Idx_{\C,M_{\bullet}}\D \cong K_{S_\bullet(\C)}.$$
\end{corollary}
\begin{proof}
Lemma \ref{lemma:Idx_can} implies that we have a canonical equivalence
$$\Idx_{\C,I}\D \cong \Idx_{\C,I^B}\D \cong \Idx_{\C,B[k]}\D.$$
To conclude the argument, we have to show that $\Idx_{\C,B[k]}\D \cong { }K_{S_k(\C)}$.  This equivalence will be shown to be induced by the exact functor
\begin{equation}\label{eqn:SkC}S_k(\C) \to \Fun_{\C}(B[k],\D),\end{equation}
sending $(0 \hookrightarrow X_1 \hookrightarrow \cdots \hookrightarrow X_k)$ to the functor $F$ in $\Fun_{\C}(B[k],\D)$, which maps the interval $[i,j]$ to the object $X_j$. We draw the resulting diagram for $k = 2$ to illustrate the idea behind the definition:
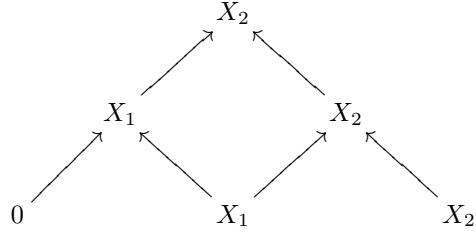
\[
\xymatrix{
& & X_2  \\
& X_1 \ar[ur] & & X_2 \ar[ul]   \\
0 \ar[ur] & & X_1 \ar[ur] \ar[ul] &  & X_2 \ar[ul] }.
\]
Alluding to Lemma \ref{lemma:inductive}, one can prove with the help of the right choice of admissible maximal tree in $B[k]$ that the induced map of index spaces is indeed an equivalence. We choose to work with the naive admissible maximal tree $T$ in $B[k]$, uniquely defined by the property that for every non-maximal element there is a unique edge in $T$ connecting it with the maximum. The image of $T$ in $B[k]^{\Delta}$, i.e. the partially ordered set obtained by identifying the basepoints $b_0,\dots,b_k$ (see Definition \ref{defi:Delta}), is also an admissible maximal tree. We can therefore apply Lemma \ref{lemma:inductive} to analyze the map of spaces
$${ }K_{\Fun_{\C}(B[k]^{\Delta},\D)} \to { }K_{\Fun_{\C}(B[k],\D)}.$$
Doing so, we obtain a commutative diagram of connective $K$-theory spectra (as in the proof of Lemma \ref{lemma:Idx_can})
\begin{equation}\label{eqn:juxtapose}
\xymatrix{
{ }\Kk_{\Fun_{\C}(B[k]^{\Delta},\D)} \ar[r] \ar[d]_{\cong} & { }\Kk_{\Fun_{\C}(B[k],\D)} \ar[d]^{\cong} \\
{ }\Kk_{\D} \oplus { }\Kk_{\C}^{\oplus E(T^{\Delta})} \ar[r]^{\alpha} & { }\Kk_{\D} \oplus { }\Kk_{\C}^{\oplus E(T)},
}
\end{equation}
where the morphism $\alpha$ agrees with the identity $1_{{ }\Kk_{\C}}$ for edges in $E(T) \setminus \{e_0,\dots,e_k\}$, and with the diagonal map
$$\Delta_{{ }\Kk_{\C}}\colon{ }\Kk_{\C} \to { }\Kk_{\C}^{\oplus(k+1)},$$
for the component $e$. This is the same map arising in the proof of Lemma \ref{lemma:Idx_can}, and we have
$$\Idx_{\C,B[k]}\D \cong \Omega^{\infty}\mathrm{cofib}({ }\Kk_{\C} \xrightarrow{\Delta_{{ }\Kk_{\C}}} { }\Kk_{\C}^{\{b_0,\dots,b_k\}}) \cong { }K^{\times k}_{\C},$$
where the last equivalence is defined as the inverse to the composition
\begin{equation}\label{eqn:i}
K^{\times k}_{\C} \to^{i} K_{\C}^{\{b_0,\dots,b_k\}} \to  \Omega^{\infty}\mathrm{cofib}({ }\Kk_{\C} \xrightarrow{\Delta_{{ }\Kk_{\C}}} \Kk_{\C}^{\{b_0,\dots,b_k\}}),
\end{equation}
where the map $i$ is the inclusion of $K^{\times k}_{\C}$ into $K_{\C}^{\{b_0,\dots,b_k\}}$ which misses the $K_{\C}^{\{b_0\}}$-factor. In particular, we see that $i$ corresponds to the map of $K$-theory spaces induced by the functor $\C^{\times k} \to \C^{\{b_0,\dots,b_k\}}$ given by the inclusion of the last $k$ factors.

Recall that we have ${ }K_{S_k(\C)} \cong { }K^{\times k}_{\C}$, with respect to the map induced by the exact functor
\begin{equation}\label{eqn:Svsk}\C^{\times k} \to S_k(\C),\end{equation}
sending
$$(X_1,\dots,X_k) \mapsto (0 \hookrightarrow X_1 \hookrightarrow X_1 \oplus X_2 \hookrightarrow \cdots \hookrightarrow X_1 \oplus \cdots \oplus X_k).$$
Composing the functors
$${ }K^{\times k}_{\C} \to { }K_{S_k(\C)} \to { }K_{\Fun_{\C}(B[k],\D)} \to \Idx_{\C,B[k]}\D \to { }K^{\times k}_{\C}$$
we obtain the identity, as can be checked on the level of exact categories: we have a commutative diagram of exact functors
\[
\xymatrix{
& & \D \times \C^{\{b_1,\dots,b_k\}} \\
\C^{\times k} \ar[r] \ar[urr] & S_k(\C) \ar[r] & \Fun_{\C}(B[k],\D), \ar[u]
}
\]
where the right vertical functor sends $F$ to $(F(b_0),F([1])/F(b_0),\ldots,F([k])/F(b_{k-1}))$. The composition of exact functors represented by the diagonal arrow is given on objects by:
\begin{align*}
    (X_1,\ldots,X_k)&\mapsto (0\into X_1\into X_1\oplus X_2\into\cdots X_1\oplus\cdots \oplus X_k)\\
    &\mapsto ([i,j]\mapsto X_1\oplus\cdots\oplus X_j)\\
    &\mapsto (0,X_1,X_2,\ldots,X_k)
\end{align*}
i.e. it is equivalent to the inclusion of the last $k$ factors in $\C^{\times k+1}$. Applying $K$-theory, and juxtaposing with \eqref{eqn:juxtapose}, we obtain a commutative diagram of spaces
\[
\xymatrix{
& & K_{\D} \times K_{\C}^{\{b_0,\dots,b_k\}} \ar[r] & \Omega^{\infty}\mathrm{cofib}({ }\Kk_{\C} \xrightarrow{\Delta_{{ }\Kk_{\C}}} { }\Kk_{\C}^{\{b_0,\dots,b_k\}})  \\
K^{\times k}_{\C} \ar[r] \ar[urr] & K_{S_k(\C)} \ar[r] & K_{\Fun_{\C}(B[k],\D)} \ar[u] \ar[r] & \Idx_{\C,B[k]}\D. \ar[u]_{\cong}
}
\]
As we observed in \eqref{eqn:i}, the composition of the arrows on the top agrees with the equivalence $\Idx_{\C,B[k]}\D \cong K^{\times k}_{\C}$.

To conclude the argument it suffices to establish the last claim. The functoriality of the index space construction guarantees that $\Idx_{\C,M_{\bullet}}\D$ is a well-defined simplicial space. Since the construction $I \mapsto I^B$ is functorial, we obtain a well-defined simplicial object $M_{\bullet}^B$, which acts as a bridge between $\Idx_{\C,M_{\bullet}}\D$ and $\Idx_{\C,B[\bullet]}\D$, i.e. according to Lemma \ref{lemma:Idx_can} we have equivalences
$$\Idx_{\C,M_{\bullet}}\D \cong \Idx_{\C,M^B_{\bullet}}\D \cong \Idx_{\C,B[{\bullet}]}\D.$$
It therefore suffices to show that $\Idx_{\C,B[{\bullet}]}\D \cong K_{S_\bullet(\C)}$ as simplicial spaces. Since the map \eqref{eqn:SkC} is clearly a map of simplicial objects in exact categories, and a map of simplicial objects is an equivalence if it is a levelwise equivalence, we may conclude the proof.
\end{proof}

\subsubsection{Rigidity of the Pre-Index Map}
In this paragraph we record a consequence of Lemma \ref{lemma:Idx_can}, which we will refer to as the \emph{rigidity} of the pre-index map. In order to formulate the result, we have to introduce a localization of the category $\bSet$.

\begin{lemma}\label{lemma:loc_rigidity}
Consider the class of morphisms $W$ in the category $\bSet$ which consists of maps $(I \to I',[k] \to^{\phi} [k'])$ such that $\phi\colon [k] \to [k']$ is an isomorphism. We denote by $\bSet[W^{-1}]$ the $\infty$-category obtained by localization at $W$. This localization is canonically equivalent to the category $\Delta$ of finite non-empty ordinals, by means of the functor
$$\mathrm{base}\colon \bSet \to \Delta,$$
which sends the pair $(I,(x_0,\dots,x_k))$ to $[k]$.
The functor $B[\bullet]\colon \Delta \to \bSet$ (Definition \ref{def:Bk}) is an inverse equivalence$$\Delta \to \bSet[W^{-1}].$$
\end{lemma}

\begin{proof}
Note that we have $\mathrm{base} \circ B[\bullet] \to^{\cong} \id_{\Delta}$.

The universal property of localization of $\infty$-categories implies that the functor $\mathrm{base}$ induces a functor
$$\widetilde{\mathrm{base}}\colon \bSet[W^{-1}] \to \Delta.$$
In particular, we obtain a natural equivalence
$$\widetilde{\mathrm{base}} \circ B[\bullet] \to^{\cong} \id_{\Delta}.$$
Similarly, we recall from the proof of Corollary \ref{cor:pre-index} that we have a natural transformation
$$\id_{\bSet} \to (-)^B\colon \bSet \to \bSet,$$
as well as $B[\bullet]\circ\mathrm{base} \to (-)^B$. Putting these two natural transformations together, we obtain a zigzag
$$\id_{\bSet} \rightarrow (-)^B \leftarrow B[\bullet] \circ \mathrm{base},$$
which induces a natural equivalence of functors
$$\id_{\bSet[W^{-1}]} \to^{\cong} B[\bullet] \circ \widetilde{\mathrm{base}}.$$
We conclude that $B[\bullet]$ and $\widetilde{\mathrm{base}}$ are mutually inverse equivalences of $\infty$-categories (in fact this shows that the $\infty$-category $\bSet[W^{-1}]$ is equivalent to a category).
\end{proof}

We use this localization to draw the following porism from the proof of Corollary \ref{cor:pre-index}.

\begin{proposition}\label{prop:rigidity}
The functor $\Idx\colon \Expair \times \bSet^{\op} \to \Spaces$, of Definition \ref{defi:Delta}, descends along the localization $\bSet \to \bSet[W^{-1}]$ of Lemma \ref{lemma:loc_rigidity}. In particular, by virtue of the equivalence
$$\bSet[W^{-1}] \cong \Delta,$$
we see that $\Idx$ induces a functor
$$\Expair  \times \Delta^{\op} \to \Spaces.$$
\end{proposition}

\begin{rmk}
The above implies that the functor $\Idx$ gives rise to a simplicial object $\Idx_{\bullet}$ in the $\infty$-category of functors $\Fun(\Expair,\Spaces)$. Corollary \ref{cor:pre-index} can be restated as
$$\Idx_{\C,\bullet}\D\cong K_{S_\bullet(\C)}.$$
\end{rmk}

\begin{proof}[Proof of Proposition \ref{prop:rigidity}]
We have seen, in Lemma \ref{lemma:Idx_can}, that every inclusion $I \subset I'$ which restricts to a bijection on basepoints induces an equivalence of index spaces
$$\Idx_{\C,I}\D \cong \Idx_{\C,I'}\D.$$
As in the proof of Corollary \ref{cor:pre-index} we observe that the zigzag of inclusions
$$I \subset I^B \supset B[\mathrm{base}(I)]$$
yields a zigzag of equivalences of index spaces. In particular, we see that the functor $\Idx$ is equivalent to $\Idx \circ B[\bullet] \circ \mathrm{base}$. In particular, it factors through the map $\mathrm{base}\colon \bSet \to \Delta$.
\end{proof}

In Subsection \ref{sub:revisit} we sketch a construction of index spaces for infinite filtered sets, using the rigidity property as main ingredient.

\subsubsection{Three Examples for the Structure of the Pre-Index Map}

In order to shed some light on the abstract constructions introduced above, we take a look at a few concrete examples. This serves a purely expository purpose, and we will only refer to the results of this paragraph to illustrate the theory. The first example is a simple lemma illustrating that the ostensible complexity of the definitions above can be avoided if $\C = \D$.

\begin{example}\label{ex:CisD}
Let $\C$ be an exact category. Then, for every based, filtered, partially ordered set $(I;x_0,\ldots,x_k)$, the pre-index map
$$|\Fun_{\C}(I,\C)^\times| \to \Idx_{\C,I}\C \cong K_{\C}^{\times k}$$
is equivalent to the map
$$F \mapsto (F(x_1) - F(x_0),\dots,F(x_{k})- F(x_{k-1})),$$
where we view $F(x_i)$ as a point in the $K$-theory space $K_{\C}$ and we use the subtraction operation stemming from the infinite loop space structure of $K$-theory spaces (which is well-defined, up to a contractible space of choices).
\end{example}

This follows directly from the next example, by setting $\D = \C$ and using that for every diagram $F \in \Fun_{\C}(I,\C)$ the maps
$F(m)/F(x_i)-F(m)/F(x_{i+1})$
and $F(x_{i+1})-F(x_i)$ are naturally homotopic (this follows from the basic properties of algebraic $K$-theory).
%Without loss of generality we may assume that $I = B[k]$. The map of the assertion is homotopic to the zero map on $\Fun_{\C}(I^\Delta,\C)$: for every diagram in this category, we have a choice of identifications $F(x_0) \cong \cdots \cong F(x_k)$. In particular, we obtain a well-defined map $\Idx_{\C,I}\C \to K_{\C}^{\times k}$. This map fits into a commutative diagram
%\[
%\xymatrix{
%K_{S_{k}\C} \ar[r] \ar[rd] & K_{\Fun_{\C}(B[k],\C)} \ar[r] \ar[d] &  \Idx_{\C,I}\C \ar[dl] \\
%& K^{\times k}_{\C}, &
%}
%\]
%where the diagonal arrow is the inverse of the equivalence of \eqref{eqn:Svsk}, it is obtained by applying $K$-theory to the exact functor $S_{k}\C \to \C^k$ given by $(0 \mapsto X_1 \mapsto \cdots \mapsto X_k) \mapsto (X_1,X_2/X_1,\dots,X_k/X_{k-1})$. The horizontal map is the one defined in \eqref{eqn:SkC}.  The right half of the commutative diagram exists (as we observed at the beginning of the proof), and it remains to construct the left half of the diagram. It suffices to check commutativity in the homotopy category of spaces (as this diagram is only a $2$-simplex).
%
%The composition of

\begin{example}
Let $I$ be a based, finite, filtered, partially ordered set such that the $k$ basepoints are pairwise distinct. We denote the unique maximal element of $I$ by $m$. Then, the pre-index map
$$|\Fun_{\C}(I,\D)^\times| \to K^{\times k}_{\C}$$
can be expressed as
$$(F(m)/F(x_0)-F(m)/F(x_1),\dots,F(m)/F(x_{k-1}) - F(m)/F(x_k)).$$
\end{example}

\begin{proof}
For the proof we recall the description of the index space $\Idx_{\C,I}$ given in terms of admissible trees (see the proof of Lemma \ref{lemma:Idx_can}). Let $T$ be the admissible tree in $\Gamma(I)$, consisting precisely of the set of edges $\{e_x\}_{x \in I}$, where $e_x$ connects the point $x$ with the maximal element $m$. As observed in the proof of Lemma \ref{lemma:Idx_can}, the infinite loop space underlying $\Idx_{\C,I}\D$, is equivalent to the cofibre of the map of connective spectra
$$\Kk_{\C}^{E(T^{\Delta})} \to^{\alpha} \Kk_{\C}^{E(T)}.$$
In the homotopy category of spectra this morphism belongs to a distinguished triangle which can be written as a sum of two distinguished triangles: the first summand is given by
$$\Kk_{\C}^{E(T) \setminus \{e_{b_0},\dots,e_{b_k}\}} \to \Kk_{E}^{E(T) \setminus \{e_{b_0},\dots,e_{b_k}\}} \to 0 \to \Sigma\Kk_{\C}^{E(T) \setminus \{e_{b_0},\dots,e_{b_k}\}}$$
and corresponds to the edges in $T$ which do not contain a base point.
The second summand is
$$\Kk_{\C} \to^{\Delta} \Kk_{\C}^{k+1} \to^{\beta} \Kk_{\C}^k \to \Sigma\Kk_{\C},$$
where $\Delta$ denotes the diagonal inclusion, and $\beta$ is given by $(x_0,\dots,x_k) \mapsto (x_0 - x_1,\dots, x_{m-1} - x_m)$.

The claim now follows from the definition of the exact functor $\Fun_{\C}(I,\D) \to \D \times \C^{\times E(T)}$ as
$$F \mapsto (F(b_0),(F(m)/F(x))_{x \in I \setminus \{m\}}),$$
where we use the identification $E(T) = I \setminus \{m\}$.
%Without loss of generality we may assume that $I = B[k]$. As in the proof of Example \ref{ex:CisD} we have to establish the existence of a commutative triangle
%\[
%\xymatrix{
%K^{\times k}_{\C} \ar[r] \ar[rd] & \Idx_{\C,I}\D \ar[d] \\
%& K^{\times k}_{\C},
%}
%\]
%where the horizontal arrow actually factors through $\Idx_{\C,I}\C$. In particular, we see that the assertion follows from Example \ref{ex:CisD}, once we have shown that for every diagram $F \in \Fun_{\C}(I,\C)$ the maps
%$F(m)/F(x_i)-F(m)/F(x_{i+1})$
%and $F(x_{i+1})-F(x_i)$ are naturally homotopic. This follows directly from the basic properties of algebraic $K$-theory.
\end{proof}

\begin{example}\label{ex:Index_homotopies}
Let $I$ be $B[2]$ with its three basepoints $b_0$, $b_ 1$, and $b_2$. It contains three copies of $B[1]$, indexed by the set of unordered pairs of distinct elements in $\{b_0,b_1,b_2\}$. We denote these inclusions by $\phi_{ij}\colon B[1] \to B[2]$. For every $F \in \Fun_\C(I,\D)$, we have a contractible space of homotopies
$$\phi_{01}^*F + \phi_{12}^*F \simeq \phi_{02}^*F$$
in $K_{\C} \cong K_{S_1(\C)} \cong \Idx_{\C,B[1]}\D$.
\end{example}

\begin{proof}
We will construct these homotopies as homotopies of loops in $K_{\C} \cong \Omega|K_{S_\bullet(\C)}|$. By Corollary \ref{cor:pre-index}, for every simplicial object $M_{\bullet}$ in $\bSet$ with $k+1$ basepoints in level $k$, we have a map of simplicial spaces
$$\left(\Fun_{\C}(M_{\bullet},\D)\right)^{\grp} \to K_{S_\bullet(\C)}.$$
We apply this observation to the degenerate simplicial object $M_{\bullet}$, which agrees with $B[k]$ for $k \leq 2$, and satisfies $M_k = B[2]$ for $k \geq 2$, with the last basepoint $x_2$ repeated $k-2$ times in $M_k$. In particular, a diagram $F$ gives rise to a $2$-simplex of the left hand side
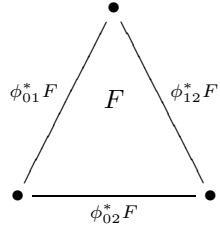
\[
\xymatrix{
& \bullet & \\
& F & \\
\bullet \ar@{-}[uur]^{\phi_{01}^*F} \ar@{-}[rr]_{\phi_{02}^*F} & & \bullet \ar@{-}[uul]_{\phi_{12}^*F}
}
\]
with boundary faces $\phi_{01}^*F$, $\phi_{12}^*F$, and $\phi_{02}^*F$. Since $K_{S_0(\C)} \cong 0$, every $1$-simplex induces an element of $\Omega|K_{S_\bullet(\C)}|$. The geometric realization of the triangle above yields a contractible space of homotopies between the loops $\phi_{01}^*F \cdot \phi_{12}^*F$ and $\phi_{02}^*F$.
\end{proof}

The existence of such a homotopy is not surprising. Indeed, passing to $K_0$, this statement amounts to the simple observation that we have the identity
$$F(x_{01})/F(x_0)-F(x_{01})/F(x_1) + F(x_{12})/F(x_1) - F(x_{12})/F(x_2) = F(x_{02})/F(x_0) - F(x_{02})/F(x_2).$$

The pre-index provides a natural contractible space of choices for this homotopy. We return to this at the end of this section.

\subsection{The Index Map for Tate Objects Revisited}\label{sub:revisit}
We now apply the generalized Waldhausen construction to produce a simplicial map
\begin{equation}\label{segalmap}
    N_\bullet\elTate(\C)^\times\to K_{S_\bullet(\C)}
\end{equation}
whose geometric realization is equivalent to the index map. For any elementary Tate object $V$, by pre-composing \eqref{segalmap} with the map
\begin{equation*}
    B_\bullet\Aut(V)\to N_\bullet\elTate(\C)^\times
\end{equation*}
we obtain a map of reduced Segal objects in $\Spaces$
\begin{equation*}
     B_\bullet\Aut(V)\to K_{S_\bullet(\C)}
\end{equation*}
which encodes the $A_\infty$-structure of the index map.

Let $\bcSet$ denote the category of (possibly infinite) filtered posets $I$, together with a choice of basepoints $(x_0,\dots,x_k) \in I^{[k]}$. Note that we do not impose the condition that the basepoints are minimal in $I$.

\begin{definition}\label{defi:tautology}
For $(I;x_0,\dots,x_k) \in \bcSet$, and $(\D,\C) \in \Expair$ we define:
\begin{itemize}
\item[(a)] $\Fun_{\C}(I,\D)$ is the exact category of functors $I \to \D$, such that $x \leq y$ in $I$ is sent to an admissible monomorphism in $\C$ with cokernel in $\D$.
\item[(b)] $\Fun^{\ast}_{\C}(I,\D)$ as the colimit of exact categories $\colim_{I'}\Fun_{\C}(I',\D)$,
\item[(c)] $\Idx_{\C,I}\D$ as the colimit of spaces $\colim_{I'} \Idx_{\C,I'}\D,$
\end{itemize}
where $I'$ ranges over the filtered category of finite based sets $(I';x_0,\dots,x_k)$ together with a map of based sets $(I';x_0,\dots,x_k) \to (I;x_0,\dots,x_k)$, which corresponds to $\id_{[k]}$.
\end{definition}

Just as in the case of finite based sets, these constructions are sufficiently natural in the pair $(\D,\C)$ and the based set $I$. This follows from Lurie's functoriality of (co)limits result \cite[Proposition 4.2.2.7]{Lurie:bh}, applied to the following set-up: we let
$S$ be the (nerve) of the category $\bcSet$, and $Y \to S$ be the constant cartesian fibration with fibre given by the $\infty$-category $\Fun(\Expair,\Spaces)$. We consider the diagram $K \to S$ given by (the nerve of) the category $\bSet/\bcSet$ together with the obvious functor to $\bcSet$. The functor $\bSet \to \Fun(\Expair,\Spaces)$ of Definition \ref{defi:Delta} gives rise to a functor $K \to Y$ which belongs to a commutative diagram
\[
\xymatrix{
K \ar[r] \ar[rd] & Y \ar[d] \\
& S.
}
\]
According to \cite[Proposition 4.2.2.7]{Lurie:bh} there exists a functor $S = \bcSet \to^{\Idx} \Fun(\Expair,\Spaces)$, such that for every $I \in \bcSet$ we have an equivalence $\Idx_I({D,\C}) \cong \colim_{I'/I} \Idx_{\C,I'}\D$, where $I' \in \bSet$.
We record these observations in the lemma below.

\begin{lemma}
There exist functors $$\Fun \colon \bcSet^{\op} \times \Expair \to \Expair,$$ and $$\Fun^\ast \colon \bcSet^{\op} \times \Expair \to \Expair,$$ and $$\Idx\colon \bcSet^{\op} \times \Expair \to \Spaces$$ which are compatible with Definition \ref{defi:tautology}. Moreover there are natural transformations $$\Fun^{\grp} \to (\Fun^\ast)^{\grp} \to \Idx$$ extending the canonical one for finite based sets.
\end{lemma}

Since the category we are taking the colimit over in Definition \ref{defi:tautology} is co-filtered, and for a morphism $I' \to I''$ (inducing the identity on base points) the induced map of index spaces
$$\Idx_{\C,I''}\D \to \Idx_{\C,I'}\D$$
is an equivalence by Lemma \ref{lemma:Idx_can}, we are taking an inverse limit over a co-filtered system of equivalences. Hence, we have a canonical equivalence of index spaces
$\Idx_{\C,I}\D \cong \Idx_{\C,I'}\D.$ This implies at once that the rigidity property (Proposition \ref{prop:rigidity}) holds as well for objects in $\bcSet$.

\begin{definition}\label{defi:tautology2}
Let $\Gr_{\bullet}(\C)^{\grp}$ denote the Grothendieck construction of the functor $\elTate(\C)^{\grp} \to \sSet$, which sends $V \in \elTate(\C)^{\grp}$ to the simplicial set of (unordered) tuples of lattices in $\Gr(V)$, i.e. an $n$-simplex in $\Gr_{\bullet}(\C)^{\grp}$ is given by the data $(V;L_0,\dots,L_n)$, where $V \in \elTate(\C)^{\grp}$, and each $L_i$ denotes a lattice in $V$.
\end{definition}

We will construct a morphism
$$\Gr_{\bullet}(\C)^\times_{\bullet} \to K_{S_\bullet(\C)},$$
which informally speaking, sends $(V;L_0,\dots,L_k)$ to $(\Gr(V);L_0,\dots,L_k) \in \bcSet$, and then computes the index of the tautological diagram $\Gr(V) \to \Pro(\C)$, which sends $L \in \Gr(V)$ to the corresponding Pro-object.
To make this rigorous we begin with a technical observation.

\begin{rmk}\label{rmk:Grothendieck_S}
The Grothendieck construction (for simplicial {\em sets}) turns a simplicial set $M_{\bullet}$ into a category $\widetilde{M}_{\bullet} \to \Delta^{\op}$ over the opposite category of finite non-empty ordinals. We have a canonical equivalence
$$M_{\bullet} \cong \colim_{\widetilde{M}_{\bullet}/\Delta^{\op}} \{\bullet\},$$
where we take a fibrewise colimit (in the $\infty$-category of spaces, see 4.3.1 in \cite{Lurie:bh}) on the left hand side over the constant, singleton-valued diagram indexed by $\widetilde{M}$.
\end{rmk}

We apply this remark to the simplicial set $\Gr_{\bullet}(V)$, where $V$ is a Tate object, in order to define the following morphism.

\begin{definition}\label{defi:theta_2}
For $V\in\elTate(\C)^{\grp}$, consider the canonical map $$\{\Fun_{\C}(\Gr(V)^{\grp},\Pro(\C))^{\grp}\}_{\widetilde{\Gr}_{\bullet}(V)} \to \Idx_{\C,\bullet}\Pro(\C) \cong K_{S_\bullet(\C)}.$$ Pre-composing it with the map
$$\Gr_{\bullet}(V) \to \Fun_{\C}(\Gr(V),\D)^{\grp}$$
which sends $(L_0,\dots,L_k) \in \Gr_k(V)$ to the tautological $\C$-diagram $\Gr(V) \to \Pro(\C)$ of the based set $(\Gr(V),L_0,\dots,L_k)$, we obtain a natural transformation of diagrams indexed by $\elTate(\C)^{\grp}$:
$$\{ \Gr_{\bullet}(V)\}_{\elTate(\C)^{\grp}} \to \{K_{S_\bullet(\C)}\}.$$
By virtue of the universal property of colimits (since the right hand side is a constant diagram), we obtain a morphism
$$\Gr_{\bullet}(\C)^{\grp} \to K_{S_\bullet(\C)}.$$
\end{definition}

\subsection{Comparison}\label{sec:compare}

It remains to verify compatibility of Definition \ref{defi:theta_2} with the index map.

\begin{proposition}\label{prop}
There exists a commutative diagram
\[
\xymatrix{
\Gr_{\bullet}^{\le}(\C)^{\grp} \ar[rd]_{\Index} \ar[r] & \Gr_{\bullet}^{\grp}  \ar[d] \\
 & K_{S_\bullet(\C)}
}
\]
in the $\infty$-category of simplicial diagrams of spaces.
\end{proposition}

The proof rests on the following technical lemma.

\begin{lemma}\label{lemma}
Let $S \in \bcSet$ be a based filtered set with basepoints $(x_0,\dots,x_n)$. We assume that
\begin{enumerate}
\item[(a)] we have $x_0 \leq \cdots \leq x_n$,
\item[(b)] for $s \in S$ we have $s \leq x_i$ for $i = 0,\dots,n$ implies $s= y_j$ for some $j$ with $0 \leq j \leq i$,
\item[(c)] there exists $y \in S$, such that $y \geq x_i$ for $i = 0,\dots,n$,
\item[(d)] there is a surjective morphism $S \to^{\phi} S'$ of based filtered sets, which contracts the elements $(x_0,\dots,x_n)$ to a single point $x \in S'$, and is an equivalence on $S \setminus \{x_0,\dots,x_n\}$.
\end{enumerate}
Then the functor $\phi^*\colon \Fun_{\C}(S',\D) \to \Fun_{\C}(S,\D)$ is a left $s$-filtering embedding in the sense of \cite[Definition 1.5]{MR2079996}.
\end{lemma}

\begin{proof}
Let $S' \to S$ be the unique section to $\phi$ sending $x$ to $x_n$. There is a natural transformation $\phi^*s^* \hookrightarrow \id$, which is object-wise an admissible monomorphism. Moreover we have a natural isomorphism $s^* \phi^* \simeq (\phi \circ s)^* \simeq \id$. We therefore conclude that $s^*$ is the left adjoint to $\phi^*$, and that $\phi^*$ is fully faithful.

Given an admissible short exact sequence $X \hookrightarrow Y \twoheadrightarrow \phi(Z)$ with $Z \in \phi^*(\Fun_{\C}(S',\D))$ we may apply the exact functor $\phi^* s^*$ to obtain a short exact sequence $\phi^* s^* X \hookrightarrow \phi^* s^* Y \twoheadrightarrow \phi(Z)$ in the essential image of $\phi$. The natural transformation $\phi^* s^* \to \id$ yields that $\phi^*$ is left special.

It remains to show that $\phi^*$ is left special: every morphism $\phi(X) \to Z$ factors through an admissible monomorphism $\phi(X) \to \phi(Y) \hookrightarrow Z$. This is possible since one can define $Y=s^*Z$, and consider the admissible monomorphism $\phi^*s^*(Z) \hookrightarrow Z$.
\end{proof}

Theorem 2.10 in \cite{MR2079996} implies the following.

\begin{corollary}\label{cor}
For $S$ and $S'$ as in Lemma \ref{lemma}, there is a natural morphism $$K_{\Fun_{\C}(S,\D)/\phi^*\Fun_{\C}(S',\D)} \to \Idx_{S,\C}\D$$ and in particular we have a commutative diagram of spaces
\[
\xymatrix{
\Fun_{\C}(S,\D)^{\grp} \ar[r] \ar[rd] & \left(\Fun_{\C}(S,\D)/\phi^*\Fun_{\C}(S',\D)\right)^{\grp} \ar[d] \\
& \Idx_{S,\C}\D.
}
\]
\end{corollary}

\begin{proof}[Proof of Proposition \ref{prop}]
By Definition \ref{defi:theta_2}, the composition $\Gr_\bullet^\le(\C)^\times_\bullet \to \Gr_\bullet(\C)^\times_\bullet \to K_{S_\bullet(\C)}$ is equivalent to the levelwise colimit of the map of constant diagrams
$$\{\ast\}_{\widetilde{\Gr^{\le}_{\bullet}(V)}/\Delta^{\op}} \to \{\Fun_{\C}(\Gr(V),\Pro(\C))^{\grp}\}_{\widetilde{\Gr^{\le}_{\bullet}(V)}/\Delta^{\op}} \to  \{\Idx_{\C,\Gr_{\bullet}}\Pro(\C)\}_{\widetilde{\Gr^{\le}_{\bullet}(V)}/\Delta^{\op}},$$
where $\ast$ is sent to the canonical admissible diagram $\Gr(V) \to \Pro(\C)$ sending $L \in \Gr(V)$ to the Pro object $L$.

Next we introduce a variant of the construction $S^B$. Let $A[n]$ be the filtered poset $\{(x,y) \in [n] \times [n] | x \leq y\}$, ordered lexicographically. It is clear that this defines a co-simplicial object in the category of filtered posets. For a based poset $(S;x_0,\dots,x_n)$, we define $S^A$ to be the pushout of posets
$$S^A = S \cup_{[n]} A[n]$$
along the map $[n] \to S$ given by $i \mapsto x_i$, and $[n] \to A[n]$ given by the diagonal. As basepoints we choose $a_i=(i,0) \in A[n]$ for $0 \leq i \leq n$.

In the following we use the notation $L_0 \subset \cdots \subset L_k$ to denote an element in $\Gr_k^{\leq}(V)$. The tautological $\Gr(V)$-diagram extends to $\Gr(V)^A$, by sending the interval $(x,y)$ to $L_x$. For the resulting $A[n]$-subdiagram, we have an admissible epimorphism in $\Fun_{\C}(A[n],\Pro(\C))$, to the admissible $A[n]$-diagram obtained by restricting the admissible $[n]$-diagram
\begin{equation}\label{eqn:restrict}0 \hookrightarrow L_1/L_0 \hookrightarrow \cdots \hookrightarrow L_n/L_0
\end{equation}
 the morphism of filtered posets $A[n] \to [n]$ given by the projection to the first component.

The kernel of the admissible epic relating the two diagrams lies in $\Fun_{\C}(A([n])',\Pro(\C))$. By Corollary \ref{cor} the above colimit is therefore equivalent to the colimit of constant diagrams
$$\{\ast\}_{\widetilde{\Gr^{\le}_{\bullet}(V)}/\Delta^{\op}} \to S_{\bullet}\C^{\grp} \to \{\Fun_{\C}(A[\bullet],\Pro(\C))^{\grp}\}_{\widetilde{\Gr^{\le}_{\bullet}(V)}/\Delta^{\op}} \to  \{\Idx_{\C,A[{\bullet}]}\Pro(\C)\}_{\widetilde{\Gr^{\le}_{\bullet}(V)}/\Delta^{\op}}.$$
This shows that the resulting $A[n]$-subdiagram lies in the image of the functor $$S_\bullet(\C) \to \Fun_{\C}(A[n],\Pro(\C)).$$ Assuming that this functor is compatible with the equivalence $\Idx_{\C,\bullet}\Pro(\C) \cong K_{S_\bullet(\C)}$, we use that the morphism $$\Gr^{\le}_{\bullet}(V)^{\grp} \to K_{S_\bullet(\C)}$$
factors through the canonical map $\Gr^{\le}_{\bullet} \to S_{\bullet}(\C)^{\grp}$ to conclude the proof.

In order to establish the required compatibility, we denote by $T[n]$ the based filtered set, given by $n+1$ basepoints $x_0,\dots,x_n$ and a unique maximal point $m$. There are natural maps $T[n] \to A[n]$ and $T[n] \to B[n]$. The commutative diagram
\[
\xymatrix{
& \Fun_{\C}(B[\bullet],\D) \ar[rd] & \\
S_{\bullet}\C \ar[ru] \ar[rd] &  & \Fun_{\C}(T[\bullet],\D) \\
& \Fun_{\C}(A[\bullet],\D) \ar[ru] &
}
\]
of exact categories commutes. It induces a commutative diagram
\[
\xymatrix{
& \Idx_{\C,B[\bullet]}\D \ar[rd]^{\simeq} & \\
S_{\bullet}\C \ar[ru]^{\simeq} \ar[rd] &  & \Idx_{\C,T[\bullet]}\D \\
& \Idx_{\C,A[\bullet]}\D \ar[ru]_{\simeq} &
}
\]
of equivalences by virtue of rigidity (Proposition \ref{prop:rigidity}).
\end{proof}

Choose a representative $V$ for every isomorphism class of elementary Tate objects, and select a lattice $L \in \Gr(V)$. This allows one to define a pseudo-simplicial map of simplicial groupoids
\[N_{\bullet}\elTate(\C)^{\grp} \simeq \bigsqcup_{V \in \elTate(\C)/\text{iso}} B_{\bullet} \Aut(V) \to^{\Lc} \Gr_{\bullet}(\C)^{\grp},\]
where we view $B_\bullet\Aut(V)$ as a discrete simplicial groupoid (i.e. no nontrivial morphisms), and where $\Lc$ sends an $n$-simplex $(g_1,\dots,g_n) \in B_n\Aut(V)$ to $(L,g_1L,\dots,g_n\cdots g_1 L)$. Note that this map is simplicial away from $d_0$, i.e. $d_i\Lc=\Lc d_i$ for $i>0$, and $s_i\Lc=\Lc s_i$ for all $i$. The component at $\bar{g}:=(g_1,\ldots,g_n)$ of the natural isomorphism $\Lc d_0\to^{\alpha} d_0\Lc$ is given by
\begin{equation*}
    \alpha_{\bar{g}}=(g_1,g_2g_1g_2^{-1},\ldots,g_n\cdots g_1 g_2^{-1}\cdots g_n^{-1})\colon (L,g_2L,\ldots,g_n\cdots g_2 L)\to (g_1L,g_2g_1L,\ldots,g_n\cdots g_1 L).
\end{equation*}
One can check directly that $d_0\alpha_{\bar{g}}\circ \alpha_{d_0\bar{g}}=\alpha_{d_1\bar{g}}$ as required for $(\Lc,\alpha)$ to give a pseudo-simplicial map.

Post-composing this map with $\Gr_{\bullet}(\C)^{\times} \to K_{S_{\bullet}\C}$ of Definition \ref{defi:theta_2} we obtain a morphism of Segal objects
\begin{equation}\label{eqn:segal}N_{\bullet}\elTate(\C)^{\grp} \to K_{S_{\bullet} \C}.\end{equation}

\begin{theorem}
The map of $A_{\infty}$-objects $\Aut(V) \to K_{\C}$ encoded by \eqref{eqn:segal}, agrees with the natural $A_{\infty}$-structure obtained by applying $\Omega$ to the map $B\Aut(V) \to BK_{\C}$.
\end{theorem}

\begin{proof}
We have a morphism of simplicial objects $B_{\bullet}\Aut(V) \to \Gr_{\bullet}(\C)^{\times} \to K_{S_{\bullet}\C}$. We claim that the forgetful map $\Gr_\bullet(\C)^\times\to \elTate(\C)^\times$ is an equivalence after geometrically realizing. Indeed, by its definition as a Grothendieck construction, we have an equivalence of spaces
\[
    |\Gr_\bullet(\C)^\times|\simeq \colim_{\elTate(\C)^\times} |\Gr_\bullet(V)|
\]
where the colimit on the right hand side is the colimit in the $\infty$-category of spaces of the functor
\begin{align*}
    \Gr_\bullet(-)\colon\elTate(\C)^\times&\to\sSet\\
    V&\mapsto \Gr_\bullet(V).
\end{align*}
Let $\{\bullet\}$ denote the constant diagram
\begin{align*}
    \{\bullet\}\colon\elTate(\C)^\times&\to\sSet\\
    V&\mapsto \Delta^0
\end{align*}
and consider the map to the constant diagram $\Gr_\bullet(-)\to\{\bullet\}$. After geometrically realizing, this gives a pointwise equivalence of diagrams; indeed, for any $V\in\elTate(\C)^\times$, the simplicial set $\Gr_\bullet(V)$ is 0-coskeletal, which implies that the map $\Gr_\bullet(V)\to\Delta^0$ is a trivial fibration. Therefore
\begin{align*}
     |\Gr_\bullet(\C)^\times| &\simeq \colim_{\elTate(\C)^\times} |\Gr_\bullet(V)|\\
     &\simeq \colim_{\elTate(\C)^\times}\{\bullet\}\simeq |\elTate(\C)^\times|
\end{align*}
as claimed.

We now show that the geometric realization of the map $\Lc$ is homotopy inverse to this map. Denote by $B_\bullet^{css}\elTate(\C)^\times$ the complete Segal space associated to the groupoid $\elTate(\C)^\times$, i.e.
\[
    B_n^{css}\elTate(\C)^\times:=\Fun([n],\elTate(\C)^\times)^\times.
\]
Recall the adjunctions
\[
    p_j^\ast\colon\sSet\leftrightarrows\ssSet\colon \iota_j^\ast
\]
for $j=1,2$ (see Appendix, \S \ref{app}). Observe that the inclusion of horizontal and vertical 0-simplices give canonical maps
\[
    p_j^\ast N_\bullet\elTate(\C)^\times\to B_\bullet^{css}\elTate(\C)^\times.
\]
for $j=1,2$. For $j=1$, this is an equivalence of complete Segal spaces, by \cite[Theorem 4.11]{JT} (it is the co-unit for the Quillen equivalence $p_1^\ast\dashv \iota_1^\ast$, see Appendix, \S \ref{app}). By Lemma~\ref{lemma:key}, these two inclusions become equivalent after applying the functor
\[
    t_!\colon\ssSet\to\sSet
\]
(see Appendix, \S \ref{app}).  By \cite[Theorem 4.12]{JT}, $t_!$ is a Quillen equivalence from the model structure for complete Segal spaces to the model structure for quasi-categories. By Corollary~\ref{cor:key}, we conclude that the two inclusions, viewed as a zig-zag from $\elTate(\C)^\times$ to itself, are canonically equal to the identity.

The pseudo-simplicial map $\Lc$ extends (along the inclusion of vertical 0-simplices $N_\bullet\elTate(\C)^\times\to B_\bullet^{css}\elTate(\C)^\times$) to a pseudo-simplicial map of simplicial groupoids
\[
    B_\bullet^{css}\elTate(\C)^\times\to^{\Lc}\Gr_\bullet(\C)^\times
\]
where concretely, $\Lc$ is given on objects by the formula above. On morphisms, $\Lc$ is given by
\begin{align*}
    \Lc&((g_1,\ldots,g_n)\to^{(h_0,\ldots,h_n)} (h_1g_1h_0^{-1},\ldots,h_ng_nh_{n-1}^{-1}))\\
    &=(L,g_1L,\ldots,g_n\cdots g_1L)\to^{(1,h_1g_1h_0^{-1}g_1^{-1},\ldots,h_n(g_n\cdots g_1)h_0^{-1}(g_n\cdots g_1)^{-1})} (L,h_1g_1h_0^{-1}L,\ldots,h_ng_n\cdots g_1h_0^{-1}L).
\end{align*}
One can check that $\alpha$ as above defines a natural transformation $\alpha\colon d_0\Lc\to\Lc d_0$. By inspection, the composition
\[
    p_2^\ast N_\bullet\elTate(\C)^\times\to B_\bullet^{css}\elTate(\C)^\times\to^{\Lc}\Gr_\bullet(\C)^\times\to \elTate(\C)^\times
\]
is the identity. By the above, the maps
\[
    p_j^\ast N_\bullet\elTate(\C)^\times\to B_\bullet^{css}\elTate(\C)^\times\to^{\Lc}\Gr_\bullet(\C)^\times
\]
are canonically equivalent for $j=1,2$; in particular, the map
\[
    \Lc\colon N_\bullet\elTate(\C)^\times\to \Gr_\bullet(\C)^\times
\]
is canonically inverse to the equivalence
\[
    \Gr_\bullet(\C)^\times\to\elTate(\C)^\times
\]
as claimed.

According to Proposition \ref{prop}, the geometric realization of the chain of maps
\[
    N_\bullet\Aut(V)\to^{\Lc}\Gr_\bullet(\C)^\times\to K_{S_\bullet(\C)}
\]
is therefore equivalent to the index map
$$B\Aut(V) \to^{\Index} BK_{\C}.$$
Theorem \cite[5.2.6.15]{Lurie:ha} implies that geometric realization induces an equivalence between the $\infty$-category of Segal objects $X_{\bullet}$ with $X_0$ contractible, and the $\infty$-category of connected pointed spaces. This shows that the $A_{\infty}$-structure we defined above, agrees with the one which naturally lives on the index map.
\end{proof}

\section{Appendix}\label{app}
In this appendix, we recall basic facts about complete Segal spaces and groupoids.\medskip

Let $\C$ be a category. Let $B_\bullet^{css}\C$ be the associated complete Segal space, i.e.
\begin{equation*}
    B_n^{css}\C=|\Fun([n],\C)^\times|.
\end{equation*}
For definiteness of notation, we view a complete Segal space as a bisimplicial set, with the simplicial direction horizontally, and the spaces given by the columns, e.g.
\begin{equation*}
    (B_\bullet^{css}\C)_{m,n}:=N_n \Fun([m],\C)^\times.
\end{equation*}
Recall the Quillen equivalence of Joyal-Tierney \cite[Section 2 and Theorem 4.12]{JT}
\[
    t_!\colon\ssSet\leftrightarrows \sSet\colon t^!
\]
from the Rezk model structure (for complete Segal spaces) on $\ssSet$ to the Joyal model structure (for quasi-categories) on $\sSet$. By definition,
\[
    t_!([m]\times[n]):=\Delta^m\times\Delta'[n]
\]
where $\Delta'[n]$ denotes the nerve of the groupoid freely generated by the category $[n]$. In general, $t_!$ is the left Kan extension of $t_!$ along the Yoneda embedding, while $t^!$ is the functor
\[
    (t^!X)_{m,n}:=\hom_{\sSet}(\Delta^m\times\Delta'[n],X).
\]
Recall also the projections and inclusions
\[
    \iota_j\colon \Delta\to\Delta\times\Delta\colon p_j
\]
where $\Delta$ is the ordinal category, and $j=1,2$.  We denote the associated functors
\[
    p_j^\ast \colon \sSet\to \ssSet \colon \iota_j^\ast.
\]
Then $p_j^*\dashv \iota_j^*$ for $j=1,2$. By \cite[Theorem 4.11]{JT}, $p_1^\ast\dashv\iota_1^\ast$ is also a Quillen equivalence from the Rezk model structure (for complete Segal spaces) on $\ssSet$ to the Joyal model structure (for quasi-categories) on $\sSet$.

\begin{lemma}
    For a category $\C$, with nerve $N\C$, there is a natural isomorphism of bisimplicial sets \[B_\bullet^{css}\C\cong t^! N\C.\]
\end{lemma}
\begin{proof}
    By definition,
    \begin{align*}
        (B^{css}_\bullet\C)_{m,n}&:=N_n\Fun([m],\C)^\times\\
        &=\ob\Fun([m]\times\Delta'[n],\C).\intertext{Further, because the nerve preserves products and gives a fully faithful embedding of the category of categories into the category of simplicial sets, the righthand side is naturally isomorphic to}
        &\cong\hom_{\sSet}(\Delta^m\times\Delta'[n],N\C)=(t^!N\C).
    \end{align*}
\end{proof}

\begin{lemma}\label{lemma:pre}
    For a category $\C$, with core $\C^\times$, there exist natural isomorphisms
    \begin{align*}
        N\C        &\cong \iota_1^\ast t^! N\C\text{,}\\
        N\C^\times &\cong \iota_2^\ast t^! N\C\text{,}\\
        N\C^\times &\cong t_! p_2^\ast N\C^\times\text{.}
    \end{align*}
\end{lemma}
\begin{proof}
    The first statement is immediate from the definitions, and in fact holds for any simplicial set $X$. For the second, by definition,
    \begin{align*}
        (\iota_2^\ast t^! N\C)_n&=\hom_{\sSet}(\Delta[0]\times\Delta'[n],N\C)\\
        &\cong\hom_{\sSet}(\Delta'[n],N\C)\\
        &\cong\hom_{\sSet}(\Delta'[n],N\C^\times)\\
        &\cong \hom_{\sSet}(\Delta^n,N\C^\times)=N_n\C^\times.
    \end{align*}
    The second claim follows from the first by the uniqueness of adjoints. Concretely, we restrict the adjunction
    \[
        t_!p_2^\ast\dashv \iota_2^\ast t^!
    \]
    to the full subcategories of (nerves of) groupoids in $\sSet$ and (Rezk nerves of) groupoids in $\ssSet$. Then the above shows that after restricting to groupoids, $\iota_2^\ast t^!\cong 1$; therefore, the left adjoints, i.e. $t_!p_2^\ast$ and $1$ are also isomorphic.
\end{proof}

Let $\varepsilon_t\colon t_!t^!\Rightarrow 1$ denote the co-unit of the adjunction $t_!\dashv t^!$. For a bisimplicial set $X_{\bullet,\bullet}$, let $\varepsilon_2\colon p_2^\ast\iota_2^\ast X\into X$ denote the inclusion of horizontal 0-simplices, i.e. the co-unit of the adjunction $p_2^\ast \dashv \iota_2^\ast$.
\begin{lemma}\label{lemma:key}
    Let $\mathcal{G}$ be a groupoid. Then the compositions
    \[
        N\mathcal{G}\to^\cong t_!p_2^\ast N\mathcal{G}\to^\cong t_!p_2^\ast\iota_2^\ast t^!N\mathcal{G}\to^{t_!\varepsilon_2} t_!t^! N\mathcal{G}=t_! B_\bullet^{css}\mathcal{G}\to^{\varepsilon_t} N\mathcal{G}
    \]
    and
    \[
        N\mathcal{G}\to^\cong t_!p_1^\ast N\mathcal{G}\to^\cong t_! p_1^\ast\iota_1^\ast t^!N\mathcal{G}\to^{t_!\varepsilon_1} t_!t^!N\mathcal{G}=t_!B_\bullet^{css}\mathcal{G}\to^{\varepsilon_t}N\mathcal{G}
    \]
    are the identity. In particular, the two maps
    \[
         N\mathcal{G}\to^\cong t_!p_j^\ast N\mathcal{G}\to^\cong t_!p_j^\ast\iota_j^\ast t^!N\mathcal{G}\to^{t_!\varepsilon_j} t_!t^! N\mathcal{G}
    \]
    for $j=1,2$ are canonically equivalent.
\end{lemma}
\begin{proof}
    For the first, by the adjunction $t_!\dashv t^!$, it suffices to prove that
    \[
        p_2^\ast N\mathcal{G} \to^\cong  p_2^\ast \iota_2^\ast t^!N\mathcal{G}\to^{\varepsilon_{2,t^!}} t^!N\mathcal{G}\to^1 t^!N\mathcal{G}
    \]
    is the inclusion of horizontal 0-simplices. But this follows immediately from Lemma \ref{lemma:pre}. Similarly, for the second, it suffices to prove that
    \[
        p_1^\ast N\mathcal{G}\to^\cong p_1^\ast \iota_1^\ast t^!N\mathcal{G}\to^{\varepsilon_!} t^!N\mathcal{G}\to^1 t^!N\mathcal{G}
    \]
    is the inclusion of vertical 0-simplices. But this follows by inspection. For the last claim, the two maps are each (strict) inverses of the weak equivalence $\varepsilon_t$; the claim follows.
\end{proof}

\begin{corollary}\label{cor:key}
    Let $\mathcal{G}$ be a groupoid. Then $t_!$ takes the zig-zag of weak equivalences
    \[
        p_2^\ast N\mathcal{G}\to t^!N\mathcal{G} \leftarrow p_1^\ast N\mathcal{G}
    \]
    to the identity.
\end{corollary}
\begin{proof}
    By Lemma~\ref{lemma:key}, $t_!$ applied to both maps gives $\varepsilon_t^{-1}$.  This is equivalent to the identity via the map of spans
    \[
        \xymatrix{
           N\mathcal{G} \ar[r]^{\varepsilon_t^{-1}} \ar[d]^1 & t_!t^! N\mathcal{G} \ar[d]^{\varepsilon_t} & N\mathcal{G} \ar[l]_{\varepsilon_t^{-1}} \ar[d]^1\\
           N\mathcal{G} \ar[r]^1 & N\mathcal{G} & N\mathcal{G}. \ar[l]_1
        }
    \]
\end{proof}
\bibliographystyle{amsalpha}
\bibliography{master}
\end{document}